\documentclass{mcom-l}

\newtheorem{theorem}{Theorem}[section]
\newtheorem{lemma}[theorem]{Lemma}
\newtheorem{proposition}[theorem]{Proposition}

\theoremstyle{definition}

\theoremstyle{remark}
\newtheorem{remark}[theorem]{Remark}

\numberwithin{equation}{section}

\usepackage{bm}

	\allowdisplaybreaks

    \newcommand{\hf}{\frac{1}{2}}

    \newcommand{\nabh}{\nabla_{\! h}}

    \newcommand{\nrm}[1]{\left\| #1 \right\|}

    \def\0{\mbox{\boldmath $0$}}

\newcommand{\eipx}[2]{\left[ #1 , #2 \right]_{\rm x}}
\newcommand{\eipy}[2]{\left[ #1 , #2 \right]_{\rm y}}
\newcommand{\eipz}[2]{\left[ #1 , #2 \right]_{\rm z}}

\usepackage[normalem]{ulem}
\usepackage{color}

\begin{document}

\title[Global-in-time energy stability analysis for ETDRK scheme]
{Global-in-time energy stability analysis for the exponential time differencing Runge--Kutta scheme for the phase field crystal equation}


\author{Xiao Li}
\address{Key Laboratory of Mathematics and Complex Systems, Ministry of Education and School of Mathematical Sciences, Beijing Normal University, Beijing 100875, China}
\email{lixiao@bnu.edu.cn}
\thanks{}

\author{Zhonghua Qiao}
\address{Department of Applied Mathematics, The Hong Kong Polytechnic University, Hung Hom, Kowloon, Hong Kong}
\email{zqiao@polyu.edu.hk}
\thanks{The second author was partially supported by the Hong Kong Research Council GRF grants 15325816 and 15300417.}

\author{Cheng Wang}
\address{Mathematics Department, The University of Massachusetts, North Dartmouth, MA 02747, USA}
\email{cwang1@umassd.edu}
\thanks{The third author was partially supported by NSF grants DMS-2012269 and DMS-2309548.}

\author{Nan Zheng}
\address{Department of Applied Mathematics, The Hong Kong Polytechnic University, Hung Hom, Kowloon, Hong Kong}
\email{znan2017@163.com}
\thanks{}

\subjclass[2020]{Primary 35K30, 35K55, 65M06, 65M12, 65T40}

\date{}

\dedicatory{}

\keywords{Phase field crystal equation,
exponential time differencing,
Runge--Kutta numerical scheme,
energy stability,
eigenvalue estimate}

\begin{abstract}
The global-in-time energy estimate is derived for the second-order accurate exponential time differencing Runge--Kutta (ETDRK2) numerical scheme to the phase field crystal (PFC) equation, a sixth-order parabolic equation modeling crystal evolution. 
To recover the value of stabilization constant, some local-in-time convergence analysis has been reported, and the energy stability becomes available over a fixed final time. In this work, we develop a global-in-time energy estimate for the ETDRK2 numerical scheme to the PFC equation by showing the energy dissipation property for
any final time. An \textit{a~priori} assumption at the previous time step, combined with a single-step $H^2$ estimate of the numerical solution, is the key point in the analysis. Such an $H^2$ estimate recovers the maximum norm bound of the numerical solution at the next time step, and then the value of the stabilization parameter can be theoretically justified. This justification ensures the energy dissipation at the next time step, so that the mathematical induction can be effectively applied, by then the global-in-time energy estimate is accomplished. This paper represents the first effort to theoretically establish a global-in-time energy stability analysis for a second-order stabilized numerical scheme in terms of the original free energy functional. The presented methodology is expected to be available for many other Runge--Kutta numerical schemes to the gradient flow equations.
\end{abstract}

\maketitle


\section{Introduction}
	
The phase field crystal (PFC) equation has become a very powerful model to describe crystal dynamics at the atomic scale in space and on diffusive scales in time~\cite{elder02}. The elastic and plastic deformations, as well as multiple crystal orientations and defects, have been appropriately incorporated in this approach. This physical model has been widely used in the numerical simulation of many related microstructures~\cite{provatas07}, such as epitaxial thin film growth~\cite{elder04}, grain growth~\cite{stefanovic06}, eutectic solidification~\cite{elder07}, and dislocation formation and motion~\cite{stefanovic06}, etc. The PFC equation is a gradient flow model, and the phase variable stands for a coarse-grained temporal average of the number density of atoms; also see the related derivation of dynamic density functional theory \cite{backofen07, marconi99, provatas10}.  
In the PFC formulation, $u: \Omega\subset \mathbb{R}^3\rightarrow \mathbb{R}$ is the atom density variable, and the free energy is given by~\cite{elder02, elder04, swift77}
	\begin{equation*}
E (u) = \int_\Omega \Big( \frac{1}{4} u^4 +\frac{1-\varepsilon}{2} u^2 - | \nabla u |^2 +\frac{1}{2} (\Delta u)^2 \Big) \, \mathrm{d}\bm{x}  ,
	\end{equation*}
in which the parameter $0 < \varepsilon < 1$ measures a deviation from the melting temperature. 
Moreover, we assume a periodic boundary condition for the sake of brevity, and an extension to the case of homogeneous Neumann boundary condition is straightforward.

Meanwhile, an equivalent free energy functional is used to simplify the analysis~\cite{shin2016}:
\begin{equation*}
  E (u) = \int_\Omega \, \Big( \frac14 u^4 - \frac{\varepsilon}{2} u^2 + \frac12 ( ( I + \Delta ) u )^2 \Big) \, \mathrm{d}\bm{x} .
\end{equation*}
The PFC equation becomes the associated $H^{-1}$ gradient flow of the free energy, i.e.,
\begin{equation}
  u_t = \Delta \mu ,  \quad \mu:= \delta_u E = u^3 - \varepsilon u + ( I + \Delta )^2 u .
  \label{equation-PFC-1}
\end{equation}

Many numerical works have been reported for the PFC equation in the existing literature. Of course, a theoretical justification of energy stability has always been used as a mathematical check for a numerical scheme to gradient flows, since it plays a crucial role in the long time simulation. There have been extensive works of energy stability and convergence analysis for various numerical schemes to the PFC equation, as well as the modified PFC and square PFC models, including the first-order algorithms~\cite{wang10c, wang11a, wise09} and the second-order accurate ones~\cite{baskaran13a, baskaran13b, cheng2019d, dong18, hu09, WangM2021}, etc.

On the other hand, it is observed that, most energy stable numerical schemes for the PFC equation~\eqref{equation-PFC-1} involve an implicit treatment of the nonlinear term, which comes from the convexity structure of the free energy functional. Such an implicit treatment leads to a nonlinear numerical solver, which makes the implementation process very challenging. In addition, most existing works on second- and higher-order accurate schemes for the PFC equation and the modified models correspond to a multi-step algorithm, so that the reported energy stability is in terms of a modified energy functional, which is the original free energy combined with a few numerical correction terms. Such a modified energy estimate leads to a uniform-in-time bound for the original energy functional, while the original energy dissipation property has not been theoretically justified.

To obtain the stability estimate for the original energy functional, some Runge--Kutta (RK) numerical approaches have attracted increasing attention in recent years. For example, a combination of the convex splitting technique and the implicit-explicit (IMEX) RK idea leads to a convex splitting RK (CSRK) framework for gradient flow equations~\cite{Shin2017} based on a  resemblance condition. Such a CSRK framework gives a three-stage, second-order accurate nonlinear implicit scheme with a dissipation property  for the original energy. In practical computations, this three-stage RK numerical algorithm leads to three nonlinear solvers at each time step, making it even more expensive than the multi-step nonlinear numerical schemes~\cite{dong18, hu09}, in which only one nonlinear solver is needed. As an alternate approach, a linear IMEX-RK scheme is proposed in~\cite{Fu2022}, in which linear stabilization terms are used in the numerical design, and   the unconditional energy stability is proved under a global Lipschitz condition assumption. However, a theoretical justification of such a global Lipschitz condition has not been available due to the lack of estimates for the numerical solution in the maximum norm, particularly when a nonlinear term appears in the gradient flow equation. 

Meanwhile, the exponential time differencing (ETD)-based numerical approach has been another popular effort to solve nonlinear parabolic PDEs, in which an exact integration of the linear and positive definite part of the PDE is used, combined with certain explicit approximations to the temporal integral of the nonlinear and concave terms~\cite{Suli14, Beylkin98, Cox02, DaiH2023, Hochbruck10, Hochbruck11, Ju14, Ju15a, Ju15b, WangX16, Zhu16}. The energy stability analysis has been reported for a few multi-step ETD schemes~\cite{chen20a, chen20b, cheng2019c, Ju18} in their applications to various gradient flow models, while such a stability analysis has always been associated with a modified energy because of the multi-step nature. More recently, a second-order accurate ETD Runge--Kutta (ETDRK2) numerical scheme is studied for the PFC equation~\eqref{equation-PFC-1} in \cite{LiX2023c}. In this numerical approach, the right-hand side is decomposed into two parts: the stabilized diffusion part $L_\kappa$, consisting of the physical diffusion and artificial diffusion terms, while the nonlinear and the concave artificial terms are combined as the remaining part $f_\kappa$. Subsequently, an exact ETD integration is applied to the stabilized diffusion part $L_\kappa$, and a  specific explicit update of the nonlinear part $f_\kappa$ is used 
  to ensure the desired accuracy order is satisfied. For such an ETDRK2 numerical scheme, a careful estimate reveals an energy stability in terms of the original free energy functional (in comparison with the modified energy stability for many multi-step numerical schemes) under the condition of  a global Lipschitz constant. In turn, the artificial regularization parameter $\kappa$ is required to be greater than a constant, dependent on the maximum norms of the numerical solution at the previous and current time steps and at the intermediate time stage.

In the existing work~\cite{LiX2023c}, a local-in-time convergence analysis is performed for the ETDRK2 numerical scheme in the $\ell^\infty (0, T; \ell^\infty)$ norm, so that the distance between the exact and numerical solutions stays bounded for a fixed final time. Then, the $\ell^\infty$ bound of the numerical solution can be derived by the associated bound of the exact solution plus a fixed constant, as long as the convergence estimate is valid. With such an $\ell^\infty$ bound for the numerical solution, a theoretical analysis of the energy stability forms a close argument. On the other hand, such an energy stability analysis is only local-in-time, since all the error estimates for a nonlinear PDE have always contained a convergence constant of the form $\mathrm{e}^{C T}$. In turn, such a convergence constant has an exponential growth as the final time becomes larger, and a theoretical justification of the distance between the exact and numerical solutions is   no longer valid for a fixed time step size and spatial mesh in the  long-time simulation.

In this article, we provide a global-in-time energy estimate for the proposed ETDRK2 scheme to the PFC equation~\eqref{equation-PFC-1}, where the energy dissipation property is valid for any final time. Based on the established result, the key point of such an analysis is to derive a uniform-in-time bound for the numerical solution in the maximum norm. To this end, we make an \textit{a~priori} assumption of   decreasing energy at the previous time step, so that the discrete $H^2$ and $\ell^\infty$ bounds of the numerical solution become available at the current step. Subsequently, the numerical system at the intermediate time stage and the next time step is carefully analyzed, which leads to a single-step $H^2$ estimate. More precisely, two sub-stages are formulated at each Runge--Kutta  stage, and nonlinear analysis in the Fourier pseudo-spectral space is undertaken, in which a careful eigenvalue estimate plays an important role. The derived  single-step $H^2$ estimate recovers the $\ell^\infty$ bounds of the numerical solution at the intermediate stage and the next time step, so that a theoretical justification of the artificial stabilization parameter  $\kappa$ becomes available. Such an evaluation of $\kappa$ ensures the energy dissipation at the next time step, so that the mathematical induction argument can be effectively applied, and thus the global-in-time energy estimate is accomplished.

In fact, the reported framework for the global-in-time energy stability estimate is expected to be applicable to a class of gradient flow models, such as Cahn--Hilliard equation, epitaxial thin film growth, and other related gradient equations with non-quadratic free energy expansion. This scientific idea can also be applied to a class of high-order Runge--Kutta numerical schemes, including the ETDRK, the exponential and exponential-free Runge--Kutta numerical algorithms with any accuracy order, as long as the energy stability can be proved under a condition of a global Lipschitz constant.
Moreover, for a wide class of Runge--Kutta numerical schemes for the gradient flow model,  the reported theoretical technique can aslo be applied to derive the uniform-in-time bound of the numerical solution under the associated functional norm (required by the global Lipschitz constant in the energy stability estimate), hence the global-in-time energy estimate can also be theoretically justified.

The rest of this paper is organized as follows. In Section \ref{sec:numerical scheme}, we review the ETDRK2 numerical scheme and present a few preliminary estimates. A global-in-time energy stability analysis is provided in Section~\ref{sec:energy stability}. 
Some concluding remarks are presented in Section \ref{sec:conclusion}.

\section{The numerical scheme and a few preliminary estimates} \label{sec:numerical scheme}

\subsection{The finite difference spatial discretization}  

The numerical approximation on the computational domain $\Omega = (0,L)^3$ is taken into consideration with periodic boundary condition. The standard centered finite difference approximation is applied with $\Delta x = \Delta y = \Delta z = h = \frac{L}{N}$, in which $N \in\mathbb{N}$ is the spatial mesh resolution. In particular, $f_{i,j,k}$ represents the numerical value of $f$ at the regular numerical mesh points $( i h, j h, k h )$, and the discrete space ${\mathcal C}_{\rm per}$ is introduced as
	\[
{\mathcal C}_{\rm per} := \left\{ f = (f_{i,j,k} ) \,|\,  f_{i,j,k} = f_{i+\alpha N,j+\beta N, k+\gamma N}, \ \forall \, i,j,k,\alpha,\beta,\gamma\in \mathbb{Z} \right\}.
	\]
In turn, the discrete difference operators are evaluated at $( (i + \frac12) h , j h , k h)$, $( i h, ( j+\frac12) h, k h)$ and $( i h , j h , (k+\frac12)h )$, respectively:
	\begin{align*}
& 
D_x f_{i+\hf,j,k} := \frac{1}{h} (f_{i+1,j,k} - f_{i,j,k}  ), \quad
D_y f_{i,j+\hf,k} := \frac{1}{h} (f_{i,j+1,k} - f_{i,j,k}  ) ,
	\\
& 
D_z f_{i,j,k+\hf} := \frac{1}{h} (f_{i,j,k+1} - f_{i,j,k}  ) .
	\end{align*}
For a vector function $\vec{f} = ( f^x , f^y , f^z)^T$ with $f^x$, $f^y$, $f^z$ evaluated at $( (i+\frac12)h, j h , k h)$, $( i h, (j+\frac12) h, k h)$, $( i h , j h , ( k+\frac12)h )$, respectively, the corresponding average and difference operators at the staggered mesh points are defined as follows:
	\begin{align*}
& a_x f^x_{i, j, k} := \frac{1}{2} \big(f^x_{i+\hf, j, k} + f^x_{i-\hf, j, k} \big),	 \quad
 d_x f^x_{i, j, k} := \frac{1}{h}\big(f^x_{i+\hf, j, k} - f^x_{i-\hf, j, k} \big),
          \\
 & a_y f^y_{i,j, k} := \frac{1}{2} \big(f^y_{i,j+\hf, k} + f^y_{i,j-\hf, k} \big),	 \quad
d_y f^y_{i,j, k} := \frac{1}{h} \big(f^y_{i,j+\hf, k} - f^y_{i,j-\hf, k} \big),
	\\
& a_z f^z_{i,j,k} := \frac{1}{2} \big(f^z_{i, j,k+\hf} + f^z_{i, j, k-\hf} \big),   \quad
 d_z f^z_{i,j, k} := \frac{1}{h} \big(f^z_{i, j,k+\hf} - f^z_{i, j,k-\hf} \big) .
	\end{align*}
In turn,  the discrete divergence turns out to be
\begin{equation*}
\nabla_h\cdot \big( \vec{f} \big)_{i,j,k} = ( d_x f^x )_{i,j,k}  + ( d_y f^y )_{i,j,k} + ( d_z f^z )_{i,j,k} .
\end{equation*}
In particular, if $\vec{f} = \nabla_h \phi = ( D_x \phi , D_y \phi, D_z \phi)^T$ for certain scalar grid function $\phi$, the corresponding divergence becomes
\begin{equation*}
  ( \Delta_h \phi )_{i,j,k} =  \nabla_h\cdot \big( \nabla_h \phi \big)_{i,j,k} = d_x\left( D_x \phi \right)_{i,j,k}  + d_y\left( D_y \phi \right)_{i,j,k} + d_z\left( D_z \phi \right)_{i,j,k} .
\end{equation*}

For two cell-centered grid functions $f$ and $g$, the discrete $L^2$ inner product and the associated $\ell^2$ norm are defined as
\begin{equation*}
 \left\langle f , g\right\rangle := h^3 \sum_{i,j,k=1}^N \, f_{i,j,k} g_{i,j,k} , \quad
    \| f \|_2 := (\left\langle f , f\right\rangle )^\frac12 .
\end{equation*}
In turn, the mean zero space is introduced as $\mathring{\mathcal C}_{\rm per}:=\big\{ f \in {\mathcal C}_{\rm per} \, \big| \, \overline{f} :=  \frac{1}{| \Omega|} \langle f , 1 \rangle = 0 \big\}$. Similarly, for two vector grid functions $\vec{f} = ( f^x , f^y , f^z )^T$ and $\vec{g} = ( g^x , g^y , g^z )^T$ with $f^x$ ($g^x$), $f^y$ ($g^y$), $f^z$ ($g^z$) evaluated at $( (i+\frac12)h, j h , (k+\frac12) h)$, $( i  h, (j+\frac12) h, k h)$, $( i h , j h , ( k+\frac12)h )$, respectively, the corresponding discrete inner product becomes
	\begin{align*}
	  &
   \langle \vec{f} , \vec{g} \rangle : = \eipx{f^x}{g^x}	+ \eipy{f^y}{g^y} + \eipz{f^z}{g^z}, 	
\\
  & \eipx{f^x}{g^x} := \langle a_x (f^x g^x) , 1 \rangle , \, \, \,
   \eipy{f^y}{g^y} := \langle a_y (f^y g^y) , 1 \rangle , \, \, \,
   \eipz{f^z}{g^z} := \langle a_z (f^z g^z) , 1 \rangle.
	\end{align*}	
In addition to the $\ell^2$ norm, the discrete maximum norm is introduced as  $\nrm{f}_\infty := \max_{1\le i,j,k\le N}\left| f_{i,j,k}\right|$.
Moreover, the discrete $H_h^1$ and $H_h^2$ norms are defined as 
\begin{align*}
  &
\nrm{ \nabla_h f}_2^2 : = \langle \nabh f , \nabh f \rangle = \eipx{D_x f}{D_x f} + \eipy{D_y f}{D_y f} +\eipz{D_z f}{D_z f},
\\
  &
  \nrm{f}_{H_h^1}^2 : =  \nrm{f}_2^2+ \nrm{ \nabla_h f}_2^2 ,  \quad
  \| f \|_{H_h^2}^2 :=  \| f \|_{H_h^1}^2 + \| \Delta_h f \|_2^2 .
\end{align*}
The summation-by-parts formulas are recalled in the following lemma whose detailed proof can be found in~\cite{guo16, wang11a, wise10, wise09}, etc.

	\begin{lemma}
	\label{lemma1} 
For any $\psi, \phi \in {\mathcal C}_{\rm per}$ and
any $\vec{f}$, the following summation-by-parts formulas are valid:
	\begin{align*}
  &
\langle \psi , \nabla_h\cdot\vec{f} \rangle = - \langle \nabla_h \psi ,  \vec{f} \rangle, \quad
 \langle \psi, \Delta_h \phi \rangle = - \langle \nabla_h \psi ,  \nabla_h\phi \rangle ,  \quad
 \langle  \psi, \Delta_h^2 \phi \rangle =  \langle \Delta_h \psi ,  \Delta_h \phi \rangle ,
\\
  &
  \langle \Delta_h \psi, \Delta_h^2 \phi \rangle = - \langle \nabla_h \psi ,
  \nabla_h \Delta_h^2 \phi \rangle ,  \quad
  \langle \Delta_h^3 \psi, \Delta_h^2 \phi \rangle = - \langle \nabla_h \Delta_h^2 \psi ,
  \nabla_h \Delta_h^2 \phi \rangle .
	\end{align*}
	\end{lemma}
	
In addition, the following $\phi$-functions are introduced to facilitate the numerical formulation:
\begin{equation}
  \phi_0 (a) = {\rm e}^{-a} , \quad
  \phi_1 (a) = \frac{1 - {\rm e}^{-a}}{a} , \quad
  \phi_2 (a) = \frac{a - ( 1 - {\rm e}^{-a})}{a^2} , \quad a > 0 . 
  \label{phi-1}
\end{equation}
The following result will be used in subsequent analysis, and its proof has been provided in~\cite{cheng2019c}.
\begin{lemma}  
\label{lem:lem 1}
{\rm (1)} $\phi_i (x)$ is decreasing, for $i=0, 1, 2$; 

{\rm (2)} $0 < \phi_1 (x) \le 1$, $0 < \phi_2 (x) \le \frac12$, and $0 < \frac{\phi_2 (x)}{\phi_1 (x)} \le 1$, 
 $\forall\, x > 0$.
\end{lemma}

\subsection{The numerical scheme}

The space-discrete problem of~\eqref{equation-PFC-1} is to find $u: \, [0, +\infty) \to {\mathcal C}_{\rm per}$ that
\begin{equation}
  \frac{\mathrm{d} u}{\mathrm{d}t} = \Delta_h (u^3) - \varepsilon \Delta_h u + \Delta_h ( I + \Delta_h )^2 u .
  \label{space-discrete-1}
\end{equation}
Given a constant $\kappa >0$, we add and then subtract the term $\kappa \Delta_h u$ on the right-hand side of~\eqref{space-discrete-1} and rearrange the terms into linear and nonlinear parts:
\begin{equation}
  \frac{\mathrm{d}u}{\mathrm{d}t} = - L_\kappa u + f_\kappa (u) ,   \quad
  L_\kappa = - \Delta_h \big( ( I + \Delta_h )^2 + \kappa I \big) , \, \, \,
  f_\kappa (u) = \Delta_h ( u^3) - ( \varepsilon + \kappa ) \Delta_h u .
   \label{space-discrete-2}
\end{equation}
By the variation-of-constants formula, the exact solution to~\eqref{space-discrete-2}  is given by
\begin{equation*}
  u (t^{n+1}) = {\rm e}^{- \tau L_\kappa} u (t^n) + \int_0^\tau \, {\rm e}^{- (\tau -s) L_\kappa} f_\kappa (u (t^n +s)) \, \mathrm{d}s ,
\end{equation*}
where $\{t^n=n\tau\}_{n\ge0}$ with $\tau>0$ represents the set of nodes partitioning the time interval $[0,+\infty)$.
\begin{flushleft}
	The ETDRK2 scheme has been proposed in~\cite{LiX2023c}, consists of two stages at each time step:
\end{flushleft}
\begin{subequations}
\label{ETDRK2}
\begin{align}
  u^{n+1} & = \phi_0 (\tau L_\kappa) u^n + \tau \Big( ( \phi_1 (\tau L_\kappa) - \phi_2 (\tau L_\kappa)) f_\kappa (u^n) + \phi_2 (\tau L_\kappa) f_\kappa (\tilde{u}^{n+1})  \Big)  \label{ETDRK2-1}
\\
  & =
  \tilde{u}^{n+1} + \tau \phi_2 (\tau L_\kappa) ( f_\kappa (\tilde{u}^{n+1}) - f_\kappa (u^n) ) ,
  \nonumber
\\
  \tilde{u}^{n+1} & = \phi_0 (\tau L_\kappa) u^n + \tau \phi_1 (\tau L_\kappa) f_\kappa (u^n)  .
   \label{ETDRK2-2}
\end{align}
\end{subequations}
Meanwhile, a discrete version of the energy functional is defined as
\begin{equation*}
  E_h (u) = \frac14 \langle u^4 , 1 \rangle + \frac{1 - \varepsilon}{2} \| u \|_2^2
  - \| \nabla_h u \|_2^2 + \frac12 \| \Delta_h u \|_2^2 ,  \quad \forall \,u \in {\mathcal C}_{\rm per} .
\end{equation*}
For the ETDRK2 scheme~\eqref{ETDRK2}, the following result has been proved in \cite{LiX2023c}.

\begin{lemma} \cite{LiX2023c}
\label{thm: energy stab-1}
Under the condition that
\[
\kappa \ge \frac{3 M_0^2 - \varepsilon}{2}, ~ \textrm{  where } M_0 = \max ( \| u^n \|_\infty , \| \tilde{u}^{n+1} \|_\infty , \| u^{n+1} \|_\infty ),\]
the numerical solution $\left\{ u^n \right\}_{0 \le n \le N_T}$ generated by the ETDRK2 scheme~\eqref{ETDRK2} satisfies $E_h (u^{n+1}) \le E_h (u^n)$.
\end{lemma}

In the existing work \cite{LiX2023c}, a local-in-time convergence analysis is performed, in the $\| \cdot \|_\infty$ norm, to justify the parameter $M_0$. This in turn determines the value of $\kappa$, so that the energy stability analysis becomes available in Lemma~\ref{thm: energy stab-1}. On the other hand, it is observed that, the convergence constant contains an exponential growth term in time, due to the nonlinearity of the PDE.  Hence, the $\| \cdot \|_\infty$ bound of the numerical solution  is just a local-in-time result  and such an energy stability estimate turns out to be a local-in-time analysis,  since the $\| \cdot \|_\infty$ bound of the numerical solution is only valid  local-in-time by using the convergence analysis approach.


\subsection{Some related operators in the Fourier space}

To facilitate the global-in-time energy stability analysis, we introduce the linear operators
\begin{align*}
  G_h & =  \phi_1 (\tau L_\kappa) = ( \tau L_\kappa )^{-1} ( I - {\rm e}^{-\tau L_\kappa} ) ,
\\
  G^{(1)}_h & =  \phi_2 (\tau L_\kappa) = ( \tau L_\kappa)^{-1} ( I - (\tau L_\kappa)^{-1} (I - {\rm e}^{-\tau L_\kappa}) ) ,
\\
  G^{(2)}_h & =  \left( \phi_1 (\tau L_\kappa) \right)^{-1}  \phi_2 (\tau L_\kappa)
  =  ( I - {\rm e}^{-\tau L_\kappa} )^{-1}  ( I - (\tau L_\kappa)^{-1} (I - {\rm e}^{-\tau L_\kappa}) ) ,
\end{align*}
in which $\phi_j$ are defined in~\eqref{phi-1}.  Moreover,  for any $f \in {\mathcal C}_{\rm per}$ with the following discrete Fourier expansion:
\begin{equation}
  f_{i,j,k} = \sum_{\ell,m,n=-K}^K \hat{f}_{\ell, m,n} {\rm e}^{2 \pi \mathrm{i} ( \ell x_i + m y_j + n z_k)/L} ,  
  \label{Fourier-1}
\end{equation}
 the above operators can be represented as
\begin{align*}
  &
  ( G_h f )_{i,j,k} = \sum_{\ell, m,n=-K}^K  \frac{1
  - {\rm e}^{- \tau \Lambda_{\ell,m.n}} }{\tau \Lambda_{\ell,m,n}}
  \hat{f}_{\ell,m,n} {\rm e}^{2 \pi \mathrm{i} ( \ell x_i + m y_j + n z_k)/L} ,
\\
  &
  ( G^{(1)}_h f )_{i,j,k} = \sum_{\ell, m,n=-K}^K  \frac{1 - \frac{1 - {\rm e}^{- \tau \Lambda_{\ell, m,n}} }{\tau \Lambda_{\ell,m,n}} }{\tau \Lambda_{\ell,m,n}}
  \hat{f}_{\ell, m,n} {\rm e}^{2 \pi \mathrm{i} ( \ell x_i + m y_j + n z_k)/L} , 
\\
  &
  ( G^{(2)}_h f )_{i,j,k} = \sum_{\ell, m,n=-K}^K  \frac{1 - \frac{1 - {\rm e}^{- \tau \Lambda_{\ell, m,n}} }{\tau \Lambda_{\ell,m,n}} }{1 - {\rm e}^{- \tau \Lambda_{\ell,m,n}} }
  \hat{f}_{\ell, m,n} {\rm e}^{2 \pi \mathrm{i} ( \ell x_i + m y_j + n z_k)/L} , 
\end{align*}
with
\[
\Lambda_{\ell, m,n} = \big( (1 - \lambda_{\ell,m,n} )^2 + \kappa \big) \lambda_{\ell,m,n}, \quad
\lambda_{\ell,m,n} = \frac{4}{h^2} \big( \sin^2 \frac{\ell \pi h}{L} + \sin^2 \frac{m \pi h}{L} + \sin^2 \frac{n \pi h}{L} \big).
\]
Meanwhile, since all the eigenvalues  $\frac{1 - {\rm e}^{- \tau \Lambda_{\ell,m,n}} }{\tau \Lambda_{\ell,m,n}}$ are non-negative, a natural definition of $G^{(0)}_h = ( G_h)^{1/2}$ is given by
\begin{equation*}
  ( G^{(0)}_h f )_{i,j,k} = \sum_{\ell, m,n=-K}^K  \bigg( \frac{1 - {\rm e}^{- \tau \Lambda_{\ell,m,n}} }{\tau \Lambda_{\ell,m,n}} \bigg)^{\frac12}
  \hat{f}_{\ell,m,n} {\rm e}^{2 \pi \mathrm{i} ( \ell x_i + m y_j + n z_k)/L}  .
\end{equation*}
It is clear that the operator $G^{(0)}_h$ is commutative with any discrete differential operator, and the following summation by parts formula is valid:
\begin{equation*}
   \langle f , \phi_1 (\tau L_\kappa) g \rangle = \langle f , G_h g \rangle = \langle G^{(0)}_h f , G^{(0)}_h g \rangle .
\end{equation*}
Similarly, we are able to define $G^{(3)}_h = ( G_h^{(1)})^{1/2}$ and $G^{(4)}_h = ( G_h^{(2)})^{1/2}$ as
\begin{align*}
  ( G^{(3)}_h f )_{i,j,k} & = \sum_{\ell, m,n=-K}^K  \bigg( \frac{1 - \frac{1 - {\rm e}^{- \tau \Lambda_{\ell, m,n}} }{\tau \Lambda_{\ell,m,n}} }{\tau \Lambda_{\ell,m,n}}  \bigg)^{\frac12}
  \hat{f}_{\ell,m,n} {\rm e}^{2 \pi \mathrm{i} ( \ell x_i + m y_j + n z_k)/L}  , 
\\
  ( G^{(4)}_h f )_{i,j,k} & = \sum_{\ell, m,n=-K}^K  \bigg( \frac{1 - \frac{1 - {\rm e}^{- \tau \Lambda_{\ell, m,n}} }{\tau \Lambda_{\ell,m,n}} }{1 - {\rm e}^{- \tau \Lambda_{\ell,m,n}} }  \bigg)^{\frac12}
  \hat{f}_{\ell,m,n} {\rm e}^{2 \pi \mathrm{i} ( \ell x_i + m y_j + n z_k)/L}  . 
\end{align*}
The following summation-by-parts formula can be derived in the same manner:
\begin{equation*}
   \langle f , \phi_2 (\tau L_\kappa) g \rangle = \langle f , G_h^{(1)} g \rangle = \langle G^{(0)}_h f , G^{(3)}_h G^{(4)}_h g \rangle .
\end{equation*}
In addition, the following operator is introduced, which will be used in the analysis for the diffusion part:
\begin{equation}
  ( G^{(5)}_h f )_{i,j,k}  = \sum_{\ell, m,n=-K}^K  \bigg( \frac{1 - {\rm e}^{- \tau \Lambda_{\ell,m,n}} }{\tau} \bigg)^{\frac12} \lambda_{\ell,m,n}
  \hat{f}_{\ell,m,n} {\rm e}^{2 \pi \mathrm{i} ( \ell x_i + m y_j + n z_k)/L}  .
    \label{Fourier-7-1}
\end{equation}

\subsection{A few preliminary estimates}

The following preliminary estimates are needed in the subsequent analysis; the detailed proof will be provided in Appendices~\ref{appen: prop 1} and \ref{appen: prop 2}, respectively.

\begin{proposition}  \label{prop:prop 1}
  Assume that $\kappa \ge 1$. For any $f \in {\mathcal C}_{\rm per}$, we have
\begin{align}
  &
  | G^{(0)}_h \nabla_h \Delta_h f \|_2
  \le \| G^{(0)}_h \Delta_h f \|_2^\frac23 \cdot \| G^{(0)}_h \nabla_h \Delta_h^2 f \|_2^\frac13 ,
  \label{prop-1-0-5}
\\
  &
   \| \Delta_h f \|_2^2 \ge \tau \| G_h^{(5)} f \|_2^2 ,  \label{prop-1-0-3}
\\
  &
  \langle G_h L_\kappa f , \Delta_h^2 f \rangle  = \| G_h^{(5)} f \|_2^2
  \ge \frac12 \| G_h^{(0)} \nabla_h \Delta_h^2 f \|_2^2 + (\kappa -1) \| G_h^{(0)} \nabla_h  \Delta_h f \|_2^2 , \label{prop-1-0-1}
\\
  &
  \left\langle G_h L_\kappa  f , \Delta_h^2 {\rm e}^{- \tau L_\kappa} f  \right\rangle
  \ge \| G_h^{(5)} ( {\rm e}^{- \tau L_\kappa} f ) \|_2^2 ,  \label{prop-1-0-2}
\\
  &
 \| G^{(0)}_h f \|_2  \le  \| f \|_2 , \, \, \, \| G^{(3)}_h f \|_2  \le \| G^{(0)}_h f \|_2 , \label{prop-1-0-4}
\\
 &
  \| G^{(3)}_h f \|_2  \le  \frac{1}{\sqrt{2}} \| f \|_2 , \, \, \, \| G^{(4)}_h f \|_2  \le \| f \|_2 . \label{prop-1-0-6}
\end{align}
\end{proposition}

\begin{proposition}  \label{prop:prop 2}
   For any two periodic grid functions $f, \, g \in {\mathcal C}_{\rm per}$, we have
\begin{align}
  &
  \tau \langle G_h L_\kappa  f , \Delta_h^2 {\rm e}^{-\tau L_\kappa} f \rangle
   + \| \Delta_h ( g -  {\rm e}^{-\tau L_\kappa}  f )  \|_2^2 \ge \tau \| G_h^{(5)} g \|_2^2 .   \label{prop-2-0}
\end{align}
\end{proposition}

\begin{remark}
In fact, a combination of inequalities~\eqref{prop-1-0-3} and \eqref{prop-1-0-2} implies that
\begin{align*}
  &
  \tau \langle G_h L_\kappa  f , \Delta_h^2 {\rm e}^{-\tau L_\kappa} f \rangle
   + \| \Delta_h ( g -  {\rm e}^{-\tau L_\kappa}  f )  \|_2^2
\\
  & \qquad \ge
  \tau \| G_h^{(5)}  ( {\rm e}^{-\tau L_\kappa}  f ) \|_2^2
  + \tau \| G_h^{(5)} ( g -  {\rm e}^{- \tau L_\kappa} f ) \|_2^2
  \ge \frac12 \tau \| G_h^{(5)} g \|_2^2 ,
\end{align*}
for any $f, \, g \in {\mathcal C}_{\rm per}$, in which the quadratic inequality has been applied. In comparison, the derived estimate~\eqref{prop-2-0} turns out to be a more refined one.
\end{remark}

Moreover, the following inequalities will be extensively used in the nonlinear analysis.

\begin{lemma}  \label{lem:lem 2}
For any $f \in {\mathcal C}_{\rm per}$, we have
\begin{align}
   &
  \| f \|_\infty \le C_2 ( | \bar{f} | + \| \Delta_h f \|_2 ) ,  \label{lem 2-0-1}
\\
  &
  \| \nabla_h (f^3) \|_2 \le 3 \| f \|_\infty^2 \cdot \| \nabla_h f \|_2 ,  \quad
  \| \nabla_h f \|_2  \le C_3 \| \Delta_h f \|_2  ,  \label{lem 2-0-2}
\end{align}
in which the constants $C_2$ and $C_3$ are only dependent on $\Omega$, independent on $f$, $h$ and $\kappa$.
\end{lemma}

\begin{proof}
Inequality~\eqref{lem 2-0-1} comes from Lemma 3.1 in~\cite{guo16}; the technical details are skipped.

In terms of the first inequality in~\eqref{lem 2-0-2}, we begin with the following point-wise expansion, from two neighboring mesh points, $(i,j,k)$ to $(i+1,j,k)$:
\begin{equation*}
  ( D_x (f^3) )_{i+\hf, j,k} = \frac{1}{h} (f_{i+1,j,k}^3 - f_{i,j,k}^3 )
  = \Big( f_{i+1,j,k}^2 + f_{i+1,j,k} \cdot f_{i,j,k} + f_{i,j,k}^2 \Big)  ( D_x f )_{i+\hf, j,k} .
\end{equation*}
In turn, an application of discrete H\"older inequality indicates that
\begin{equation*}
  \| D_x (f^3) \|_2 \le 3 \| f \|_\infty^2 \cdot  \| D_x f \|_2.
\end{equation*}
The  corresponding estimates in the $y$ and $z$ directions can be similarly derived. This completes the proof of the first inequality in~\eqref{lem 2-0-2}.

The second inequality in~\eqref{lem 2-0-2} also comes from Lemma 3.1 in~\cite{guo16}; the technical details are skipped for the sake of brevity.
\end{proof}

\section{The global-in-time energy stability estimate}	\label{sec:energy stability}
	
In this article, we perform a direct analysis for the numerical solution, so that a uniform-in-time $H_h^2$ estimate becomes available for the numerical solution. Because of the Sobolev embedding from $H^2$ to $L^\infty$ in the 3-D space, we are able to recover the uniform-in-time value of $M_0$ and $\kappa$ arisen in Lemma~\ref{thm: energy stab-1}. This in turn gives a global-in-time energy stability estimate for the ETDRK2 scheme~\eqref{ETDRK2}.

To proceed  with the global-in-time energy stability analysis, we make an \textit{a~priori} assumption at the previous time step:
\begin{equation}
  E_h (u^n) \le E_h (u^0) := \tilde{C}_0 .  \label{a priori-1}
\end{equation}
Such an \textit{a~priori} assumption will be recovered at the next time step. With  this discrete energy assumption, the following $H_h^2$ bound of the numerical solution can be derived:
\begin{equation}
  \| \Delta_h u^n \|_2 \le \tilde{C}_1:= 2 \Big( \tilde{C}_0 + | \Omega | \Big)^\frac12  .
  \label{a priori-2}
\end{equation}
Meanwhile, it is observed that the ETDRK2 scheme~\eqref{ETDRK2} is mass conservative at a discrete level:
\begin{equation*}
  \overline{u^{n+1}} = \overline{\tilde{u}^{n+1}} = \overline{u^n} = \overline{u^0} := \beta_0 .
\end{equation*}
Now, an application of estimates~\eqref{lem 2-0-1} and \eqref{lem 2-0-2} in Lemma~\ref{lem:lem 2} yields the following nonlinear bounds at the previous time step:
\begin{align}
&   \| u^n \|_\infty \le C_2 ( | \overline{u^n} | + \| \Delta_h u^n \|_2 )
  \le C_2 ( | \beta_0 | + \tilde{C}_1 )  := \tilde{C}_2  ,   \label{a priori-3}
\\
&   \| \nabla_h ((u^n)^3) \|_2 \le
  3 \| u^n \|_\infty^2 \cdot  C_3   \| \Delta_h u^n \|_2
  \le 3 \tilde{C}_2^2 C_3 \tilde{C}_1 := \tilde{C}_3  .   \label{a priori-4}
\end{align}

\subsection{A preliminary estimate for $\| \tilde{u}^{n+1} \|_{H_h^2}$}

We aim to obtain a rough $H_h^2$ estimate for the intermediate stage numerical solution $\tilde{u}^{n+1}$. The current form~\eqref{ETDRK2-2} of the evolutionary algorithm has not indicated a clear interaction between the linear and nonlinear terms. In order to carry out the theoretical analysis in a more convenient way, we denote $u^{n+1,*}= {\rm e}^{- \tau L_\kappa} u^n$, and rewrite the evolutionary equation \eqref{ETDRK2-2} as the following two-substage system:
\begin{align}
  \frac{u^{n+1,*} - u^n}{\tau} & = - L_\kappa \phi_1 (\tau L_\kappa)  u^n ,  \label{tilde u-H2-1-1}
\\
  \frac{\tilde{u}^{n+1} - u^{n+1,*}}{\tau} & = \phi_1 (\tau L_\kappa) f_\kappa (u^n)  .  \label{tilde u-H2-1-2}
\end{align}

Taking a discrete $\ell^2$ inner product with \eqref{tilde u-H2-1-1} by $\Delta_h^2 (u^{n+1,*} + u^n)$ results in
\begin{equation}
   \left\langle u^{n+1,*} - u^n , \Delta_h^2 (u^{n+1,*} + u^n) \right\rangle
   + \tau \langle G_h L_\kappa  u^n , \Delta_h^2 (u^{n+1,*} + u^n) \rangle = 0 .
   \label{tilde u-H2-2}
\end{equation}
For the first term, an direct application of the summation-by-parts formula gives:
\begin{align*}
   \left\langle  u^{n+1,*} - u^n , \Delta_h^2 (u^{n+1,*} + u^n) \right\rangle
  &  = \left\langle  \Delta_h ( u^{n+1,*} - u^n ) , \Delta_h (u^{n+1,*} + u^n) \right\rangle
\\
  &  =
     \|  \Delta_h u^{n+1,*} \|_2^2
   - \|  \Delta_h u^n \|_2^2 .
\end{align*}
In terms of the first term appearing in the diffusion part of~\eqref{tilde u-H2-2}, an application of estimate~\eqref{prop-1-0-1} in Proposition~\ref{prop:prop 1} implies that
\begin{equation}
  \langle G_h L_\kappa u^n, \Delta_h^2 u^n \rangle  = \| G_h^{(5)} u^n \|_2^2
  \ge \frac12 \| G_h^{(0)} \nabla_h \Delta_h^2 u^n \|_2^2
  + (\kappa -1) \| G_h^{(0)} \nabla_h  \Delta_h u^n \|_2^2 .  \label{tilde u-H2-4-1}
\end{equation}
Subsequently, a combination of (\ref{tilde u-H2-2})-(\ref{tilde u-H2-4-1})  yields
\begin{equation}
  \|  \Delta_h u^{n+1,*} \|_2^2   - \| \Delta_h u^n \|_2^2
   +  \tau ( \| G_h^{(5)} u^n \|_2^2
   + \langle G_h L_\kappa  u^n , \Delta_h^2 u^{n+1,*} \rangle ) = 0 .
   \label{tilde u-H2-5}
\end{equation}
Similarly, taking a discrete $\ell^2$ inner product with \eqref{tilde u-H2-1-2} by $2 \Delta_h^2 \tilde{u}^{n+1}$ leads to
\begin{equation}
   \left\langle \tilde{u}^{n+1} - u^{n+1,*}  , 2 \Delta_h^2 \tilde{u}^{n+1} \right\rangle
   = 2 \tau \langle G_h f_\kappa  (u^n) , \Delta_h^2 \tilde{u}^{n+1} \rangle .
   \label{tilde u-H2-6}
\end{equation}
The term on the left-hand side  can be expressed as
\begin{align}
  &   \left\langle  \tilde{u}^{n+1} - u^{n+1,*} , 2 \Delta_h^2 \tilde{u}^{n+1}  \right\rangle
  = 2 \left\langle  \Delta_h ( \tilde{u}^{n+1} - u^{n+1,*} ) , \Delta_h \tilde{u}^{n+1}  \right\rangle  \label{tilde u-H2-7}
\\
  & \qquad =
     \|  \Delta_h \tilde{u}^{n+1} \|_2^2  - \|  \Delta_h u^{n+1,*} \|_2^2
   + \|  \Delta_h ( \tilde{u}^{n+1} - u^{n+1,*} ) \|_2^2   .
 \nonumber
\end{align}
Subsequently, a combination of~\eqref{tilde u-H2-5} and \eqref{tilde u-H2-6}-\eqref{tilde u-H2-7} gives
\begin{align}
  &
  \|  \Delta_h \tilde{u}^{n+1} \|_2^2   - \| \Delta_h u^n \|_2^2
   + \|  \Delta_h ( \tilde{u}^{n+1} - u^{n+1,*} ) \|_2^2    \label{tilde u-H2-9}
\\
  & \qquad  +  \tau ( \| G_h^{(5)} u^n \|_2^2
  + \langle G_h L_\kappa  u^n , \Delta_h^2 u^{n+1,*} \rangle )
   = 2 \tau \langle G_h f_\kappa  (u^n) , \Delta_h^2 \tilde{u}^{n+1} \rangle .  \nonumber
\end{align}
Moreover, an application of inequality~\eqref{prop-2-0} (in Proposition~\ref{prop:prop 2}) reveals that
\begin{equation}
   \tau \langle G_h L_\kappa  u^n , \Delta_h^2 u^{n+1,*} \rangle
   + \|  \Delta_h ( \tilde{u}^{n+1} - u^{n+1,*} ) \|_2^2
   \ge \tau \| G_h^{(5)} \tilde{u}^{n+1} \|_2^2 .
    \label{tilde u-H2-10-1}
\end{equation}
Going back to~\eqref{tilde u-H2-9}, we arrive at
\begin{equation}
  \|  \Delta_h \tilde{u}^{n+1} \|_2^2   - \| \Delta_h u^n \|_2^2
   +  \tau ( \| G_h^{(5)} u^n \|_2^2 + \| G_h^{(5)} \tilde{u}^{n+1} \|_2^2 )
   \le 2 \tau \langle G_h f_\kappa  (u^n) , \Delta_h^2 \tilde{u}^{n+1} \rangle .
   \label{tilde u-H2-10-2}
\end{equation}
Meanwhile, due to~\eqref{prop-1-0-1} in Proposition~\ref{prop:prop 1}, we observe the following inequality:
\begin{align}
 &  \| G_h^{(5)} u^n \|_2^2 + \| G_h^{(5)} \tilde{u}^{n+1} \|_2^2
  \ge \frac12 ( \| G_h^{(0)} \nabla_h \Delta_h^2 u^n \|_2^2
  + \| G_h^{(0)} \nabla_h \Delta_h^2 \tilde{u}^{n+1} \|_2^2 )  \label{tilde u-H2-14-1}
\\
  & \qquad\qquad\qquad\qquad\qquad
  + (\kappa -1) ( \| G_h^{(0)} \nabla_h  \Delta_h u^n \|_2^2
  + \| G_h^{(0)} \nabla_h  \Delta_h \tilde{u}^{n+1} \|_2^2 ) . \nonumber
\end{align}

The right-hand side of~\eqref{tilde u-H2-10-2} contains two parts:
\begin{equation}
   2 \langle G_h f_\kappa  (u^n) , \Delta_h^2 \tilde{u}^{n+1} \rangle
   = 2 \langle G_h \Delta_h ((u^n)^3) , \Delta_h^2 \tilde{u}^{n+1} \rangle
    - 2 (\varepsilon + \kappa) \langle G_h \Delta_h u^n , \Delta_h^2 \tilde{u}^{n+1} \rangle .
    \label{tilde u-H2-11}
\end{equation}
The first term can be analyzed as follows
\begin{align}
  &
     2 \langle G_h \Delta_h ((u^n)^3) , \Delta_h^2 \tilde{u}^{n+1} \rangle
     = - 2 \langle G_h \nabla_h ((u^n)^3) ,  \nabla_h \Delta_h^2 \tilde{u}^{n+1} \rangle   \label{tilde u-H2-12}
\\
  & \qquad =
  - 2 \langle G^{(0)}_h \nabla_h ((u^n)^3) ,
   G^{(0)}_h \nabla_h \Delta_h^2 \tilde{u}^{n+1} \rangle   \nonumber
\\
  & \qquad \le
  2 \| G^{(0)}_h \nabla_h ((u^n)^3) \|_2 \cdot
   \| G^{(0)}_h \nabla_h \Delta_h^2 \tilde{u}^{n+1} \|_2  \nonumber
\\
  & \qquad \le
  2 \| \nabla_h ((u^n)^3) \|_2 \cdot
   \| G^{(0)}_h \nabla_h \Delta_h^2 \tilde{u}^{n+1} \|_2  \nonumber
\\
  & \qquad \le
  8 \| \nabla_h ((u^n)^3) \|_2^2
   + \frac18 \| G^{(0)}_h \nabla_h \Delta_h^2 \tilde{u}^{n+1} \|_2^2    \nonumber
\\
  & \qquad  \le 8  \tilde{C}_3^2
   + \frac18 \| G^{(0)}_h \nabla_h \Delta_h^2 \tilde{u}^{n+1} \|_2^2 ,  \nonumber
\end{align}
in which the summation-by-parts formulas, as well as the first inequality in~\eqref{prop-1-0-4} and the \textit{a~priori} estimate~\eqref{a priori-4} have been applied. The second term, a linear inner product term, can be decomposed into two parts:
the first part is estimated as
\begin{align}
  &
  - 2 ( 1+ \varepsilon) \langle G_h \Delta_h u^n , \Delta_h^2 \tilde{u}^{n+1} \rangle
     =  2 (1 + \varepsilon) \langle G_h \nabla_h u^n ,
     \nabla_h \Delta_h^2 \tilde{u}^{n+1} \rangle   \label{tilde u-H2-13-1}
\\
  & \qquad =
    2 (1 + \varepsilon) \langle G^{(0)}_h \nabla_h u^n ,
   G^{(0)}_h \nabla_h \Delta_h^2 \tilde{u}^{n+1} \rangle
   \nonumber
\\
  & \qquad   \le  2 (1 + \varepsilon) \| G^{(0)}_h \nabla_h u^n \|_2 \cdot
   \| G^{(0)}_h \nabla_h \Delta_h^2 \tilde{u}^{n+1} \|_2  \nonumber
\\
  & \qquad \le
  2 (1 + \varepsilon) \| \nabla_h (u^n) \|_2 \cdot
   \| G^{(0)}_h \nabla_h \Delta_h^2 \tilde{u}^{n+1} \|_2  \nonumber
\\
  & \qquad \le
     2 (1 + \varepsilon) C_3 \| \Delta_h u^n \|_2 \cdot
   \| G^{(0)}_h \nabla_h \Delta_h^2 \tilde{u}^{n+1} \|_2  \nonumber
\\
  & \qquad \le
  16 C_3^2 \| \Delta_h u^n \|_2^2
   + \frac14 \| G^{(0)}_h \nabla_h \Delta_h^2 \tilde{u}^{n+1} \|_2^2
     \nonumber
\\
  & \qquad  \le 16  C_3^2 \tilde{C}_1^2
   + \frac14 \| G^{(0)}_h \nabla_h \Delta_h^2 \tilde{u}^{n+1} \|_2^2 ,\nonumber
\end{align}
in which the inequality~\eqref{lem 2-0-2} is used,
and the second part is analyzed as
\begin{align}
  &
  - 2 ( \kappa -1) \langle G_h \Delta_h u^n , \Delta_h^2 \tilde{u}^{n+1} \rangle
     =  2 (\kappa -1) \langle G_h \nabla_h \Delta_h u^n ,
     \nabla_h \Delta_h \tilde{u}^{n+1} \rangle   \label{tilde u-H2-13-2}
\\
  & \qquad =
    2 (\kappa -1) \langle G^{(0)}_h \nabla_h \Delta_h u^n ,
   G^{(0)}_h \nabla_h \Delta_h \tilde{u}^{n+1} \rangle   \nonumber
\\
 & \qquad =
    (\kappa -1) ( \| G^{(0)}_h \nabla_h \Delta_h u^n \|_2^2
  +  \| G^{(0)}_h \nabla_h \Delta_h \tilde{u}^{n+1} \|_2^2 )  \nonumber
\\
& \qquad\qquad
  -  (\kappa -1) \| G^{(0)}_h \nabla_h \Delta_h ( \tilde{u}^{n+1} - u^n ) \|_2^2 )  . \nonumber
\end{align}
In particular, we notice that an additional dissipation term appears in the equality~\eqref{tilde u-H2-13-2}, and this dissipation term comes from the stabilization of $- \kappa \Delta_h u$ in the expansion~\eqref{space-discrete-2} for $f_\kappa (u)$.
As a result, a substitution of~\eqref{tilde u-H2-14-1}-\eqref{tilde u-H2-11} into \eqref{tilde u-H2-10-2} yields
\begin{align}
  &
  \|  \Delta_h \tilde{u}^{n+1} \|_2^2   - \| \Delta_h u^n \|_2^2
   +    \frac{\tau}{2} \| G_h^{(0)} \nabla_h \Delta_h^2 u^n \|_2^2
  + \frac{\tau}{8} \| G_h^{(0)} \nabla_h \Delta_h^2 \tilde{u}^{n+1} \|_2^2   \label{tilde u-H2-15}
\\
  & \qquad
  + (\kappa -1) \tau \| G^{(0)}_h \nabla_h \Delta_h ( \tilde{u}^{n+1} - u^n ) \|_2^2
   \le 8 \tau \tilde{C}_3^2 + 16 \tau C_3^2 \tilde{C}_1^2  .\nonumber
\end{align}
Consequently, the following bound becomes available for $\|  \Delta_h \tilde{u}^{n+1} \|_2$:
\begin{equation*}
  \|  \Delta_h \tilde{u}^{n+1} \|_2^2
  \le \| \Delta_h u^n \|_2^2  + 8 \tau \tilde{C}_3^2 + 16 \tau C_3^2 \tilde{C}_1^2
  \le ( 1 + 16 C_3^2 \tau ) \tilde{C}_1^2 + 8 \tau \tilde{C}_3^2 ,
\end{equation*}
in which the \textit{a~priori} estimate~\eqref{a priori-2} has been applied. Under the following $O (1)$ constraint for the time step size
\begin{equation}
  \tau \le \min \Big\{ \frac{1}{16} C_3^{-2}, \frac18 \tilde{C}_3^{-2} \Big\} ,  \label{condition-tau-1}
\end{equation}
we obtain a rough estimate for $\|  \Delta_h \tilde{u}^{n+1} \|_2$ as
\begin{equation*}
  \|  \Delta_h \tilde{u}^{n+1} \|_2^2
  \le  2 \tilde{C}_1^2  + 1 ,
  \, \, \,  \mbox{so that}  \, \,
  \|  \Delta_h \tilde{u}^{n+1} \|_2
  \le  \tilde{C}_4 := \sqrt{2 \tilde{C}_1^2  + 1} ,
\end{equation*}
where $\tilde{C}_4$ is  independent of $\kappa$. Similarly, an application of estimates~\eqref{lem 2-0-1} and \eqref{lem 2-0-2} in Lemma~\ref{lem:lem 2} yields the following nonlinear bounds at intermediate stage:
\begin{align}
&  \| \tilde{u}^{n+1} \|_\infty \le  C_2 ( | \overline{\tilde{u}^{n+1}} | + \| \Delta_h \tilde{u}^{n+1} \|_2 )
  \le C_2 ( | \beta_0 | + \tilde{C}_4 )  := \tilde{C}_5  ,   \label{a priori-5}
\\
&  \| \nabla_h ((\tilde{u}^{n+1})^3) \|_2 \le
  3 \| \tilde{u}^{n+1} \|_\infty^2 \cdot  C_3   \| \Delta_h \tilde{u}^{n+1} \|_2
  \le 3 \tilde{C}_5^2 C_3 \tilde{C}_4  := \tilde{C}_6  .   \label{a priori-6}
\end{align}
Again, both $\tilde{C}_5$ and $\tilde{C}_6$ are is independent of $\kappa$ and global-in-time
and $\tilde{C}_{5} \ge \tilde{C}_2$ since $\tilde{C}_{4} \ge \tilde{C}_1$.

By the way, the rough estimate~\eqref{tilde u-H2-15} also indicates
\begin{equation}
 (\kappa -1) \tau \| G^{(0)}_h \nabla_h \Delta_h ( \tilde{u}^{n+1} - u^n ) \|_2^2
   \le \tilde{C}_1^2 + ( 8 \tilde{C}_3^2 + 16 C_3^2 \tilde{C}_1^2 ) \tau
   \le  2 \tilde{C}_1^2 + 1 ,
   \label{a priori-6-2}
\end{equation}
under the constraint \eqref{condition-tau-1}.
This bound will be useful to the estimate in the next stage.

\subsection{A preliminary estimate for $\| u^{n+1} \|_{H_h^2}$}

In this section, we aim to obtain a rough $H_h^2$ estimate for the numerical solution $u^{n+1}$ at the next time step, determined by the second stage. Similarly, we denote $u^{n+1,*} = {\rm e}^{-\tau L_\kappa} u^n$, so that the evolutionary algorithm \eqref{ETDRK2-1} can be rewritten as a two-substage system:
\begin{align}
  \frac{u^{n+1,*} - u^n}{\tau} & = - L_\kappa \phi_1 (\tau L_\kappa)  u^n ,  \label{u n+1-H2-1-1}
\\
  \frac{u^{n+1} - u^{n+1,*}}{\tau} & =  \phi_1 (\tau L_\kappa) f_\kappa (u^n)
  + \phi_2 (\tau L_\kappa) ( f_\kappa (\tilde{u}^{n+1}) - f_\kappa (u^n) )  .  \label{u n+1-H2-1-2}
\end{align}

The equation~\eqref{u n+1-H2-1-1} is the same as~\eqref{tilde u-H2-1-1}, so that equality~\eqref{tilde u-H2-5} is still valid.

Taking a discrete $\ell^2$ inner product with (\ref{u n+1-H2-1-2}) by $2 \Delta_h^2 u^{n+1}$ leads to
\begin{align}
   \left\langle u^{n+1} - u^{n+1,*}  , 2 \Delta_h^2 u^{n+1} \right\rangle
  &   =2 \tau \langle G_h f_\kappa  (u^n) , \Delta_h^2 u^{n+1} \rangle
     - 2 \tau \langle G_h^{(1)} f_\kappa  (u^n) , \Delta_h^2 u^{n+1} \rangle  \nonumber
\\
   & \quad
     + 2 \tau \langle G_h^{(1)} f_\kappa  (\tilde{u}^{n+1}) , \Delta_h^2 u^{n+1} \rangle .
   \label{u n+1-H2-2}
\end{align}
The term on the left-hand side can be analyzed similarly as~\eqref{tilde u-H2-7} and \eqref{tilde u-H2-10-1}:
\begin{align*}
  &
   \left\langle  u^{n+1} - u^{n+1,*} , 2 \Delta_h^2 u^{n+1}  \right\rangle
   =
     \|  \Delta_h u^{n+1} \|_2^2  - \|  \Delta_h u^{n+1,*} \|_2^2
   + \|  \Delta_h ( u^{n+1} - u^{n+1,*} ) \|_2^2  , 
\\
  &
  \tau \langle G_h L_\kappa  u^n , \Delta_h^2 u^{n+1,*} \rangle
   + \|  \Delta_h ( u^{n+1} - u^{n+1,*} ) \|_2^2
   \ge \tau \| G_h^{(5)} u^{n+1} \|_2^2  ,
\end{align*}
and its combination with~\eqref{tilde u-H2-5} and \eqref{u n+1-H2-2} yields
\begin{align}
  &
  \|  \Delta_h u^{n+1} \|_2^2   - \| \Delta_h u^n \|_2^2
   +  \tau ( \| G_h^{(5)} u^n \|_2^2 + \| G_h^{(5)} u^{n+1} \|_2^2 )  \label{u n+1-H2-4}
\\
& \qquad \le
       2 \tau \langle G_h f_\kappa  (u^n) , \Delta_h^2 u^{n+1} \rangle
     + 2 \tau \langle G_h^{(1)} ( f_\kappa  (\tilde{u}^{n+1}) - f_\kappa (u^n)  ) ,
      \Delta_h^2 u^{n+1} \rangle  .\nonumber
\end{align}

The first term on the right-hand side of~\eqref{u n+1-H2-4} can be analyzed in a similar way as in~\eqref{tilde u-H2-11}-\eqref{tilde u-H2-13-2}; some technical details are skipped for the sake of brevity:
\begin{align}
  &
   2 \langle G_h f_\kappa  (u^n) , \Delta_h^2 u^{n+1} \rangle \nonumber
\\
  & \qquad   = 2 \langle G_h \Delta_h ((u^n)^3) , \Delta_h^2 u^{n+1} \rangle
    - 2 (\varepsilon + \kappa) \langle G_h \Delta_h u^n , \Delta_h^2 u^{n+1} \rangle , \nonumber
\\
  &
     2 \langle G_h \Delta_h ((u^n)^3) , \Delta_h^2 u^{n+1} \rangle  \label{u n+1-H2-5-2}
\\
  & \qquad     = - 2 \langle G_h \nabla_h ((u^n)^3) ,  \nabla_h \Delta_h^2 u^{n+1} \rangle    \nonumber
\\
  & \qquad \le
  2 \| \nabla_h ((u^n)^3) \|_2 \cdot
   \| G^{(0)}_h \nabla_h \Delta_h^2 u^{n+1} \|_2^2 \nonumber
\\
  & \qquad    \le 24  \tilde{C}_3^2
   + \frac{1}{24} \| G^{(0)}_h \nabla_h \Delta_h^2 u^{n+1} \|_2^2 ,  \nonumber
\\
  &
  - 2 ( 1+ \varepsilon) \langle G_h \Delta_h u^n , \Delta_h^2 u^{n+1} \rangle\nonumber
\\
 & \qquad  = 2 (1 + \varepsilon) \langle G^{(0)}_h \nabla_h u^n ,
   G^{(0)}_h \nabla_h \Delta_h^2 u^{n+1} \rangle   \nonumber
\\
 & \qquad \le
     2 (1 + \varepsilon) C_3 \| \Delta_h u^n \|_2 \cdot
   \| G^{(0)}_h \nabla_h \Delta_h^2 u^{n+1} \|_2^2   \nonumber
\\
 & \qquad  \le 32  C_3^2 \tilde{C}_1^2
   + \frac18 \| G^{(0)}_h \nabla_h \Delta_h^2 u^{n+1} \|_2^2 , \nonumber
\\
  &
  - 2 ( \kappa -1) \langle G_h \Delta_h u^n , \Delta_h^2 u^{n+1} \rangle \label{u n+1-H2-5-4}
\\
 & \qquad   =  2 (\kappa -1) \langle G_h \nabla_h \Delta_h u^n ,
     \nabla_h \Delta_h u^{n+1} \rangle  \nonumber
\\
 & \qquad =
    2 (\kappa -1) \langle G^{(0)}_h \nabla_h \Delta_h u^n ,
   G^{(0)}_h \nabla_h \Delta_h u^{n+1} \rangle   \nonumber
\\
 & \qquad  =
    (\kappa -1) ( \| G^{(0)}_h \nabla_h \Delta_h u^n \|_2^2
  +  \| G^{(0)}_h \nabla_h \Delta_h u^{n+1} \|_2^2 ) \nonumber
\\
& \qquad\qquad
  -  (\kappa -1) \| G^{(0)}_h \nabla_h \Delta_h ( u^{n+1} - u^n ) \|_2^2 )  ,  \nonumber
\end{align}
so that
\begin{align}
& 2 \tau \langle G_h f_\kappa  (u^n) , \Delta_h^2 u^{n+1} \rangle
\le (24  \tilde{C}_3^2 + 32  C_3^2 \tilde{C}_1^2) \tau
   + \frac{\tau}{6} \| G^{(0)}_h \nabla_h \Delta_h^2 u^{n+1} \|_2^2 \label{u n+1-H2-5-5}
\\
& \qquad\qquad\quad +  (\kappa -1) \tau ( \| G^{(0)}_h \nabla_h \Delta_h u^n \|_2^2
  +  \| G^{(0)}_h \nabla_h \Delta_h u^{n+1} \|_2^2 ) \nonumber
\\
& \qquad\qquad\quad   -  (\kappa -1) \tau ( \| G^{(0)}_h \nabla_h \Delta_h ( u^{n+1} - u^n ) \|_2^2 ) .\nonumber
\end{align}
For the second term on the right-hand side of~\eqref{u n+1-H2-4}, we begin with the following decomposition:
\begin{align}
 &  2  \langle G_h^{(1)} ( f_\kappa  (\tilde{u}^{n+1}) - f_\kappa (u^n)  ) ,
      \Delta_h^2 u^{n+1} \rangle  \label{u n+1-H2-6}
\\
  & \qquad  = 2  \langle G_h^{(1)} \Delta_h (  (\tilde{u}^{n+1})^3 ) ,
      \Delta_h^2 u^{n+1} \rangle
      - 2  \langle G_h^{(1)} \Delta_h ( (u^n)^3 ) ,
      \Delta_h^2 u^{n+1} \rangle  \nonumber
\\
  & \qquad \quad
      - 2  (\varepsilon + \kappa) \langle G_h^{(1)} \Delta_h ( \tilde{u}^{n+1} - u^n  ) ,
      \Delta_h^2 u^{n+1} \rangle . \nonumber
\end{align}
The two nonlinear inner product terms can be bounded in the same fashion as in~\eqref{tilde u-H2-12}, \eqref{u n+1-H2-5-2}, combined with the help of inequality~\eqref{prop-1-0-6} in Proposition~\ref{prop:prop 1}:
\begin{align}
  &
  2  \langle G_h^{(1)} \Delta_h (  (\tilde{u}^{n+1})^3 ) ,
      \Delta_h^2 u^{n+1} \rangle    \label{u n+1-H2-7-1}
\\
 & \qquad  = 2  \langle G_h^{(3)} G_h^{(4)} \Delta_h (  (\tilde{u}^{n+1})^3 ) ,
      G_h^{(0)} \Delta_h^2 u^{n+1} \rangle     \nonumber
\\
  & \qquad =
  - 2  \langle G_h^{(3)} G_h^{(4)} \nabla_h (  (\tilde{u}^{n+1})^3 ) ,
      G_h^{(0)} \nabla_h \Delta_h^2 u^{n+1} \rangle \nonumber
\\
  & \qquad   \le 2  \| G_h^{(3)} G_h^{(4)} \nabla_h (  (\tilde{u}^{n+1})^3 ) \|_2
      \cdot \| G_h^{(0)} \nabla_h \Delta_h^2 u^{n+1} \|_2    \nonumber
\\
  & \qquad \le
  \sqrt{2}  \| \nabla_h (  (\tilde{u}^{n+1})^3 ) \|_2
      \cdot \| G_h^{(0)} \nabla_h \Delta_h^2 u^{n+1} \|_2   \nonumber
\\
  & \qquad \le
  \sqrt{2}  \tilde{C}_6  \| G_h^{(0)} \nabla_h \Delta_h^2 u^{n+1} \|_2   \quad
  \mbox{(by the \textit{a~priori} estimate~\eqref{a priori-6})}   \nonumber
\\
  & \qquad \le
  12  \tilde{C}_6^2
   + \frac{1}{24} \| G^{(0)}_h \nabla_h \Delta_h^2 u^{n+1} \|_2^2 , \nonumber
\end{align}
and
\begin{equation}
  - 2  \langle G_h^{(1)} \Delta_h (  (u^n)^3 ) ,
      \Delta_h^2 u^{n+1} \rangle
  \le 12  \tilde{C}_3^2
   + \frac{1}{24} \| G^{(0)}_h \nabla_h \Delta_h^2 u^{n+1} \|_2^2 .  \label{u n+1-H2-7-2}
\end{equation}
Again, the linear diffusion inner product on the right-hand side of~\eqref{u n+1-H2-6} is split into two parts, and we analyze them separately:
\begin{align}
  &
  - 2  (1 +\varepsilon) \langle G_h^{(1)} \Delta_h ( \tilde{u}^{n+1} - u^n  ) ,
      \Delta_h^2 u^{n+1} \rangle   \label{u n+1-H2-8-1}
\\
& \qquad  = - 2  (1 + \varepsilon) \langle G_h^{(3)} G_h^{(4)} \Delta_h (  \tilde{u}^{n+1} - u^n ) ,
      G_h^{(0)} \Delta_h^2 u^{n+1} \rangle     \nonumber
\\
  &\qquad  =
   2  (1 + \varepsilon) \langle G_h^{(3)} G_h^{(4)} \nabla_h (  \tilde{u}^{n+1} - u^n ) ,
      G_h^{(0)} \nabla_h \Delta_h^2 u^{n+1} \rangle
      \nonumber
\\
  & \qquad \le
  \sqrt{2} (1 + \varepsilon) \| \nabla_h (  \tilde{u}^{n+1} - u^n ) \|_2
      \cdot \| G_h^{(0)} \nabla_h \Delta_h^2 u^{n+1} \|_2   \nonumber
\\
  & \qquad \le
   2 \sqrt{2} C_3  \| \Delta_h (  \tilde{u}^{n+1} - u^n ) \|_2
      \cdot \| G_h^{(0)} \nabla_h \Delta_h^2 u^{n+1} \|_2  \nonumber
\\
  & \qquad \le
  2 \sqrt{2}  C_3 ( \tilde{C}_2 + \tilde{C}_4)  \| G_h^{(0)} \nabla_h \Delta_h^2 u^{n+1} \|_2   \quad
  \mbox{(by \eqref{a priori-2} and \eqref{a priori-5})}   \nonumber
\\
  & \qquad \le
  12  C_3^2 ( \tilde{C}_2 + \tilde{C}_4)^2
   + \frac{1}{6} \| G^{(0)}_h \nabla_h \Delta_h^2 u^{n+1} \|_2^2 ,  \nonumber
\end{align}
and
\begin{align}
  &
  - 2  (\kappa -1) \langle G_h^{(1)} \Delta_h ( \tilde{u}^{n+1} - u^n  ) ,
      \Delta_h^2 u^{n+1} \rangle  \label{u n+1-H2-8-2}
\\
  & \qquad  = - 2  (\kappa -1) \langle G_h^{(3)} \Delta_h (  \tilde{u}^{n+1} - u^n ) ,
      G_h^{(3)} \Delta_h^2 u^{n+1} \rangle     \nonumber
\\
  & \qquad =
   2  (\kappa -1 ) \langle G_h^{(3)} \nabla_h \Delta_h (  \tilde{u}^{n+1} - u^n ) ,
      G_h^{(3)} \nabla_h \Delta_h u^{n+1} \rangle      \nonumber
\\
  & \qquad \le
  2  (\kappa -1 ) \| G_h^{(3)} \nabla_h \Delta_h (  \tilde{u}^{n+1} - u^n ) \|_2^2
      + \frac{\kappa -1}{2} \| G_h^{(3)} \nabla_h \Delta_h u^{n+1} \|_2^2   \nonumber
\\
  & \qquad \le
  2  (\kappa -1 ) \| G_h^{(0)} \nabla_h \Delta_h (  \tilde{u}^{n+1} - u^n ) \|_2^2
      + \frac{\kappa -1}{2} \| G_h^{(0)} \nabla_h \Delta_h u^{n+1} \|_2^2 , \nonumber
\end{align}
in which the preliminary inequality \eqref{prop-1-0-4} has been applied in the last step of \eqref{u n+1-H2-8-2}. Furthermore, to obtain a bound for the second term on the right-hand side of~\eqref{u n+1-H2-8-2}, we make use of the Cauchy inequality:
\begin{equation}
  \frac{\kappa -1}{2} \| G_h^{(0)} \nabla_h \Delta_h u^{n+1} \|_2^2
  \le ( \kappa -1 ) ( \| G_h^{(0)} \nabla_h \Delta_h u^n \|_2^2
   + \| G_h^{(0)} \nabla_h \Delta_h ( u^{n+1} - u^n ) \|_2^2 ) .
  \label{u n+1-H2-8-3}
\end{equation}
It is noticed that the second term on the right-hand side of~\eqref{u n+1-H2-8-3} can be balanced by the stabilization estimate in \eqref{u n+1-H2-5-4}. On the other hand, motivated by the Sobolev interpolation inequality (which comes from~\eqref{prop-1-0-5} in Proposition~\ref{prop:prop 1})
\begin{equation*}
  \| G^{(0)}_h \nabla_h \Delta_h u^n \|_2
  \le \| G^{(0)}_h \Delta_h u^n \|_2^\frac23 \cdot \| G^{(0)}_h \nabla_h \Delta_h^2 u^n \|_2^\frac13 ,
\end{equation*}
an application of  Young's inequality indicates that
\begin{align*}
   \| G^{(0)}_h \nabla_h \Delta_h u^n \|_2^2
&  \le
    \| G^{(0)}_h \Delta_h u^n \|_2^\frac43
   \cdot \| G^{(0)}_h \nabla_h \Delta_h^2 u^n \|_2^\frac23  \nonumber
\\
&   \le  \| \Delta_h u^n \|_2^\frac43
   \cdot \| G^{(0)}_h \nabla_h \Delta_h^2 u^n \|_2^\frac23  \nonumber
\\
  & \le
  \frac23 \alpha^{-\frac32} \| \Delta_h u^n \|_2^2
  + \frac13 \alpha^3 \| G^{(0)}_h \nabla_h \Delta_h^2 u^n \|_2^2 , \quad \forall \, \alpha > 0 .
\end{align*}
By taking $\frac13 \alpha^3 = \frac{1}{2 (\kappa -1)}$, so that $\alpha = ( \frac32 )^\frac13 (\kappa -1)^{-\frac13}$, we get
\begin{align}
   (\kappa -1) \| G^{(0)}_h \nabla_h \Delta_h u^n \|_2^2
&  \le
  \frac{2 (\kappa -1)}{3} \alpha^{-\frac32} \| \Delta_h u^n \|_2^2
  + \frac12 \| G^{(0)}_h \nabla_h \Delta_h^2 u^n \|_2^2 \label{u n+1-H2-8-6}
\\
  &  =
  \frac{2 \sqrt{6}}{9} (\kappa -1)^\frac32 \| \Delta_h u^n \|_2^2
  + \frac12 \| G^{(0)}_h \nabla_h \Delta_h^2 u^n \|_2^2 \nonumber
\\
  &  \le \frac{2 \sqrt{6}}{9} \tilde{C}_1^2 (\kappa -1)^\frac32
  + \frac12 \| G^{(0)}_h \nabla_h \Delta_h^2 u^n \|_2^2 , \nonumber
\end{align}
in which the \textit{a~priori} estimate~\eqref{a priori-2} has been applied in the last step. Subsequently, a substitution of~\eqref{u n+1-H2-8-6} into \eqref{u n+1-H2-8-3} gives
\begin{align}
\frac{\kappa -1}{2} \| G_h^{(0)} \nabla_h \Delta_h u^{n+1} \|_2^2
&   \le \frac{2 \sqrt{6}}{9} \tilde{C}_1^2 (\kappa -1)^\frac32
  + \frac12 \| G^{(0)}_h \nabla_h \Delta_h^2 u^n \|_2^2   \label{u n+1-H2-8-7}
\\
  & \quad
   + (\kappa -1) \| G_h^{(0)} \nabla_h \Delta_h ( u^{n+1} - u^n ) \|_2^2  . \nonumber
\end{align}
Meanwhile, we recall the \textit{a~priori} estimate~\eqref{a priori-6-2} at the previous stage. As a result, a substitution of \eqref{a priori-6-2} and \eqref{u n+1-H2-8-7} into \eqref{u n+1-H2-8-2} leads to
\begin{align}
  &
  - 2  (\kappa -1) \tau \langle G_h^{(1)} \Delta_h ( \tilde{u}^{n+1} - u^n  ) ,
      \Delta_h^2 u^{n+1} \rangle   \label{u n+1-H2-8-8}
\\
  & \qquad \le
  4 \tilde{C}_1^2 + 2 + \frac{2 \sqrt{6}}{9} \tilde{C}_1^2 \tau (\kappa -1)^\frac32
  + \frac{\tau}{2} \| G^{(0)}_h \nabla_h \Delta_h^2 u^n \|_2^2 \nonumber
\\
  & \qquad\quad
   + (\kappa -1) \tau \| G_h^{(0)} \nabla_h \Delta_h ( u^{n+1} - u^n ) \|_2^2 .  \nonumber
\end{align}
Consequently, a combination of~\eqref{u n+1-H2-8-1} and \eqref{u n+1-H2-8-8} gives
\begin{align}
  &
  - 2  (\varepsilon + \kappa) \tau \langle G_h^{(1)} \Delta_h ( \tilde{u}^{n+1} - u^n  ) ,
      \Delta_h^2 u^{n+1} \rangle   \label{u n+1-H2-8-9}
\\
  & \qquad \le
  4 \tilde{C}_1^2 + 2 + 12  C_3^2 ( \tilde{C}_2 + \tilde{C}_4)^2  \tau
   + \frac{\tau}{6} \| G^{(0)}_h \nabla_h \Delta_h^2 u^{n+1} \|_2^2  \nonumber
\\
  & \qquad \quad
  + \frac{2 \sqrt{6}}{9} \tilde{C}_1^2 \tau (\kappa -1)^\frac32
  + \frac{\tau}{2} \| G^{(0)}_h \nabla_h \Delta_h^2 u^n \|_2^2 \nonumber
\\
  & \qquad \quad
   + (\kappa -1) \tau \| G_h^{(0)} \nabla_h \Delta_h ( u^{n+1} - u^n ) \|_2^2 . \nonumber
\end{align}
Furthermore, a combination of \eqref{u n+1-H2-7-1}, \eqref{u n+1-H2-7-2} and \eqref{u n+1-H2-8-9} results in
\begin{align}
  &
   2  \tau \langle G_h^{(1)} ( f_\kappa  (\tilde{u}^{n+1}) - f_\kappa (u^n)  ) ,
      \Delta_h^2 u^{n+1} \rangle    \label{u n+1-H2-8-10}
\\
  & \qquad \le
    4 \tilde{C}_1^2 + 2
   + 12  ( \tilde{C}_3^2 + \tilde{C}_6^2 + C_3^2 ( \tilde{C}_2 + \tilde{C}_4)^2 ) \tau
   + \frac{\tau}{4} \| G^{(0)}_h \nabla_h \Delta_h^2 u^{n+1} \|_2^2  \nonumber
\\
  & \qquad \quad
  + \frac{2 \sqrt{6}}{9} \tilde{C}_1^2 \tau (\kappa -1)^\frac32
  + \frac{\tau}{2} \| G^{(0)}_h \nabla_h \Delta_h^2 u^n \|_2^2 \nonumber
\\
  & \qquad \quad
   + (\kappa -1) \tau \| G_h^{(0)} \nabla_h \Delta_h ( u^{n+1} - u^n ) \|_2^2 . \nonumber
\end{align}

Therefore, a substitution of \eqref{u n+1-H2-5-5} and \eqref{u n+1-H2-8-10}  into \eqref{u n+1-H2-4} yields
\begin{align*}
  &
  \|  \Delta_h u^{n+1} \|_2^2   - \| \Delta_h u^n \|_2^2
  +  \tau ( \| G^{(5)}_h u^{n+1} \|_2^2 +  \| G^{(5)}_h u^n \|_2^2 )
\\
   & \quad \le
    4 \tilde{C}_1^2 + 2
   + ( 36 \tilde{C}_3^2 + 12 \tilde{C}_6^2 + 32 C_3^2 \tilde{C}_1^2
   + 12 C_3^2 ( \tilde{C}_2 + \tilde{C}_4)^2 ) \tau
   + \frac{2 \sqrt{6}}{9} \tilde{C}_1^2 \tau (\kappa -1)^\frac32
\\
  & \quad \quad
  + (\kappa -1) \tau ( \| G^{(0)}_h \nabla_h \Delta_h u^n \|_2^2
  +  \| G^{(0)}_h \nabla_h \Delta_h u^{n+1} \|_2^2 )
\\
& \quad \quad
   + \frac{5 \tau}{12} \| G^{(0)}_h \nabla_h \Delta_h^2 u^{n+1} \|_2^2
  + \frac{\tau}{2} \| G^{(0)}_h \nabla_h \Delta_h^2 u^n \|_2^2  .
\end{align*}
We notice that the term $(\kappa -1) \tau \| G_h^{(0)} \nabla_h \Delta_h ( u^{n+1} - u^n ) \|_2^2$ has been balanced between~\eqref{u n+1-H2-5-4} and \eqref{u n+1-H2-8-10}, which   played an important role in the derivation. In addition, the two diffusion estimate terms have the following lower bounds, as given by inequality~\eqref{prop-1-0-1} in Proposition~\ref{prop:prop 1}:
\begin{align*}
  \| G_h^{(5)} u^n \|_2^2 + \| G_h^{(5)} u^{n+1} \|_2^2
 &  \ge \frac12 ( \| G_h^{(0)} \nabla_h \Delta_h^2 u^n \|_2^2
  + \| G_h^{(0)} \nabla_h \Delta_h^2 u^{n+1} \|_2^2 )
\\
  & \quad
  + (\kappa -1) ( \| G_h^{(0)} \nabla_h  \Delta_h u^n \|_2^2
  + \| G_h^{(0)} \nabla_h  \Delta_h u^{n+1} \|_2^2 ) .
\end{align*}
Then we arrive at
\begin{align*}
  &
  \|  \Delta_h u^{n+1} \|_2^2   - \| \Delta_h u^n \|_2^2
  +  \frac{\tau}{12} \| G_h^{(0)} \nabla_h \Delta_h^2 u^{n+1} \|_2^2
\\
   & \quad \le
    4 \tilde{C}_1^2 + 2
   + ( 36 \tilde{C}_3^2 + 12 \tilde{C}_6^2 + 32 C_3^2 \tilde{C}_1^2
   + 12 C_3^2 ( \tilde{C}_2 + \tilde{C}_4)^2 ) \tau
  + \frac{2 \sqrt{6}}{9} \tilde{C}_1^2 \tau (\kappa -1)^\frac32 ,
\end{align*}
so that the following bound becomes available for $\|  \Delta_h u^{n+1} \|_2$:
\begin{equation}
\begin{aligned}
  &
  \|  \Delta_h u^{n+1} \|_2^2
  \le  5 \tilde{C}_1^2  +2  + \frac{2 \sqrt{6}}{9} \tilde{C}_1^2 \tau (\kappa -1)^\frac32
  + \tilde{C}_8 \tau ,
\\
  &
  \tilde{C}_8 =  36 \tilde{C}_3^2 + 12 \tilde{C}_6^2 + 32 C_3^2 \tilde{C}_1^2
   + 12 C_3^2 ( \tilde{C}_2 + \tilde{C}_4)^2  .
\end{aligned}
    \label{u n+1-H2-9-4}
\end{equation}
Under the following $O (1)$ constraint for the time step size
\begin{equation}
  \tau \le \min \Big\{ \kappa^{-\frac32} , \frac{1}{32} C_3^{-2}, \frac{1}{36} \tilde{C}_3^{-2} ,
  \frac{1}{12} \tilde{C}_6^{-2} \Big\} ,
  \label{condition-tau-2}
\end{equation}
which is a stronger requirement than~\eqref{condition-tau-1}, a rough estimate can be derived for $\|  \Delta_h u^{n+1} \|_2$:
\begin{align*}
  &
  \|  \Delta_h u^{n+1} \|_2^2
  \le  6 \tilde{C}_1^2 + 2 + \tilde{C}_8 \tau
  \le 7 \tilde{C}_1^2 + 4 + \frac38 (\tilde{C}_2 + \tilde{C}_4 )^2 ,
  \, \, \,  \mbox{so that}
\\
  &
  \|  \Delta_h u^{n+1} \|_2
  \le  \tilde{C}_9 := \sqrt{7 \tilde{C}_1^2 + 4 + \frac38 (\tilde{C}_2 + \tilde{C}_4 )^2} . 
\end{align*}
Notice that $\tilde{C}_9$ is independent of $\kappa$ and time. Again, an application of estimate~\eqref{lem 2-0-1} in Lemma~\ref{lem:lem 2} implies the following $\| \cdot \|_\infty$ bound at time-step $t^{n+1}$:
\begin{align}
  \| u^{n+1} \|_\infty \le & C_2 ( | \overline{u^{n+1}} | + \| \Delta_h u^{n+1} \|_2 )
  \le C_2 ( | \beta_0 | + \tilde{C}_9 )  := \tilde{C}_{10}  .  \label{a priori-7}
\end{align}
Obviously, the constant $\tilde{C}_{10}$ is independent of $\kappa$ and global-in-time, moreover
$\tilde{C}_{10} \ge \tilde{C}_5$ since $\tilde{C}_{9} \ge \tilde{C}_4$.

\subsection{Justification of the artificial parameter $\kappa$ and the \textit{a~priori} assumption~\eqref{a priori-1}}

By making a comparison between the $\| \cdot \|_\infty$ bounds for $u^n$, $\tilde{u}^{n+1}$ and $u^{n+1}$,  given by~\eqref{a priori-3}, \eqref{a priori-5} and \eqref{a priori-7}, respectively, it is  evident that $\tilde{C}_{10} \ge \tilde{C}_5 \ge \tilde{C}_2$. In other words, the defined constant $\tilde{C}_{10}$ is greater than the maximum of $\| u^n \|_\infty$, $\| \tilde{u}^{n+1} \|_\infty$ and $\| u^{n+1} \|_\infty$. Subsequently, we proceed with
\begin{equation}
  \kappa = \max \Big( \frac{3 \tilde{C}_{10}^2 - \varepsilon}{2} , 1 \Big) . 
  \label{condition-kappa-1}
\end{equation}
Notice that $\kappa$ is an $O (1)$ constant, and it contains no singular dependence on any physical parameter. With this fixed choice of $\kappa$,  we can take the time step size $\tau$ satisfying~\eqref{condition-tau-2},  thereby enabling the availability of the energy stability estimate  for the ETDRK2 scheme~\eqref{ETDRK2} at the next time step, by Lemma~\ref{thm: energy stab-1}, i.e.,
\begin{equation}
  E_h ( u^{n+1} ) \le E_h (u^n) \le E_h (u^0) = \tilde{C}_0 .
  \label{a priori-8}
\end{equation}
This in turn recovers the \textit{a~priori} assumption~\eqref{a priori-1} at the next time step, so that an induction argument can be effectively applied. Therefore, we have proved the following theorem, which represents the primary theoretical result of this article.

\begin{theorem} \label{thm: energy stab-2}
With the choice of the artificial parameter in~\eqref{condition-kappa-1},  depending only on the initial energy and the domain $\Omega$, we take the time step size satisfying the $O(1)$ constraint~\eqref{condition-tau-2}. The numerical solution $\left\{ u^n \right\}_{0 \le n \le N_T}$ generated by the ETDRK2 scheme~\eqref{ETDRK2} satisfies $E_h (u^{n+1}) \le E_h (u^n)$.
\end{theorem}

\subsection{A refined estimate for $\| u^{n+1} \|_{H_h^2}$ and $\| u^{n+1} \|_\infty$}

We notice that the $H_h^2$ estimate~\eqref{u n+1-H2-9-4} for the numerical solution $u^{n+1}$, as well as the maximum norm estimate~\eqref{a priori-7}, is based on a direct analysis for the semi-implicit numerical scheme~\eqref{ETDRK2}, with the help of extensive applications of discrete Sobolev embedding. However, this estimate turns out to be too rough since we did not make use of the variational energy structure in the analysis. In fact, to obtain the energy dissipation at a theoretical level, an $\| \cdot \|_\infty$ bound of the numerical solution at the next time step has to be derived, due to the nonlinear term involved. Such a bound can only be possibly accomplished by a direct $H_h^2$ estimate without using the energy structure. Since the $\| \Delta_h u^{n+1} \|_2$ bound in~\eqref{u n+1-H2-9-4} contains a multiple factor of the $\| \Delta_h u^n \|_2$ bound in~\eqref{a priori-2}, this estimate cannot be used in the induction. On the other hand, with such a rough bound at hand, we are able to justify the artificial parameter value in~\eqref{condition-kappa-1}, so that the energy stability becomes theoretically available at the next time step. With a theoretical justification of the energy stability analysis, we are able to obtain a much sharper $\| \cdot \|_{H_h^2}$ and $\| \cdot \|_\infty$ bound for the numerical solution $u^{n+1}$.

More specifically, with the energy stability~\eqref{a priori-8} theoretically justified at the next time step, we apply a similar analysis as in~\eqref{a priori-2} and obtain
\begin{equation}
  \| \Delta_h u^{n+1} \|_2 \le \tilde{C}_1:= 2 \Big( \tilde{C}_0 + | \Omega | \Big)^\frac12  ,
  \label{a priori-9}
\end{equation}
which is a time-independent constant. Obviously, this bound turns out to be a much sharper estimate compared with the rough $H_h^2$ estimate \eqref{u n+1-H2-9-4}, since the variational energy structure has not been applied in the derivation of \eqref{u n+1-H2-9-4}. In turn, with the  aid of inequality~\eqref{lem 2-0-1} in Lemma~\ref{lem:lem 2}, a much sharper maximum norm bound for $u^{n+1}$ also becomes available:
\begin{equation}
  \| u^{n+1} \|_\infty \le  C_2 ( | \overline{u^{n+1}} | + \| \Delta_h u^{n+1} \|_2 )
  \le C_2 ( | \beta_0 | + \tilde{C}_1 )  := \tilde{C}_2  .  \label{a priori-10}
\end{equation}
Note that the $\| \cdot \|_{H_h^2}$ bound $\tilde{C}_1$ and the $\| \cdot \|_\infty$ bound $\tilde{C}_2$ are both global-in-time quantities.

\begin{remark}
In the rough $\| \cdot \|_{H_h^2}$ estimate
for the numerical solution $u^{n+1}$, if the inequality~\eqref{u n+1-H2-8-10} is derived in an alternate way, we are able to obtain $\| \Delta_h u^{n+1} \|_2^2 = \| \Delta_h u^n \|_2^2 + O (\tau ( 1+ \kappa^2))$ for a fixed  value of $\kappa$. In other words, even for the rough $\| \cdot \|_{H_h^2}$ estimate, a reasonable result can be available. The reason why the $\| \Delta_h u^{n+1} \|_2$ bound~\eqref{u n+1-H2-9-4} contains a multiple factor of $\| \Delta_h u^n \|_2$ is due to the fact that, the value of $\kappa$ has not been fixed in the rough estimate, so that we need the time step size $\tau$ to balance the quantity of $\kappa^\frac32$.

Meanwhile, all these derivations are rough estimates. With the bound~\eqref{a priori-7}, we are able to choose $\kappa$ as in~\eqref{condition-kappa-1}, so that the energy stability becomes theoretically available.  Consequently, much sharper estimates~\eqref{a priori-9} and \eqref{a priori-10} can be derived.
\end{remark}

\begin{remark}
The stabilization approach has been extensively applied in the numerical design for various gradient flow models, in which an artificial regularization term ensures the energy stability at a theoretical level while preserving the desired accuracy order. For the gradient models with automatic Lipschitz continuity for the nonlinear term, such as the no-slope-selection  thin film growth equation, the stabilization approach has been widely studied~\cite{chen12, chen20a, Hao2021, LiW18, Meng2020}, and the energy stability can be proved in a straightforward way. For the gradient models without  automatic Lipschitz continuity for the nonlinear part, a theoretical analysis of the maximum norm of the numerical solution is needed to establish the energy stability analysis. For example, a local-in-time convergence analysis and energy stability estimate has been provided for the stabilization schemes to the nonlocal Cahn--Hilliard equation~\cite{LiX2021, LiX2023b, LiX2023a}, in both the first and second accuracy orders. The local-in-time nature of these analyses comes from a lack of global-in-time regularity for the nonlocal gradient model. For the standard Cahn--Hilliard equation, a cut-off approach is applied in the pioneering work~\cite{shen10a}, while a theoretical justification of the artificial regularization parameter is available in the associated works~\cite{LiD2017, LiD2017b, LiD2016a}, in which a single-step $\ell^\infty$ analysis for the numerical solution is provided to establish the energy stability analysis.

Meanwhile, all the existing works on energy stability analysis for the multi-step numerical schemes are in terms of a modified energy functional,   consisting of the original free energy and a few numerical correction terms. In comparison, this article represents the first effort to theoretically establish a global-in-time energy stability analysis for a second-order stabilized numerical scheme, in terms of the original free energy functional, which comes from the single-step Runge--Kutta style of the numerical algorithm.
\end{remark}


\section{Concluding remarks}  \label{sec:conclusion}

A second-order accurate, exponential time differencing Runge--Kutta (ETDRK2) numerical scheme for the phase field crystal (PFC) equation is analyzed in detail, and a  global-in-time energy estimate is derived. Such an energy stability has been proven for the ETDRK2 numerical scheme to the PFC equation under a global Lipschitz constant assumption. To accomplish this goal, an \textit{a~priori} assumption is made at the previous time step, and a single-step $H^2$ estimate of the numerical solution is carefully performed. This single-step $H^2$ estimate gives a useful upper bound of the numerical solution at the next time step, in the discrete maximum norm, which in turn leads to a theoretical justification of the stabilization parameter value.   Consequently, such a justification ensures the energy dissipation at the next time step. As a result, the mathematical induction can be applied to derive a global-in-time energy estimate. In particular, the derived energy dissipation property is valid at any final time, in comparison with some existing works of local-in-time energy  estimates, which  come from a local-in-time convergence analysis.

The methodology presented in this work will provide a framework of the global-in-time energy stability analysis for a class of higher-order accurate, multi-stage RK-type numerical schemes for various gradient flow models, such as Allen--Cahn/Cahn--Hilliard equation, epitaxial thin film growth, and other related gradient equations with non-quadratic free energy expansion. As long as the energy stability can be  proven under a global Lipschitz condition, a similar theoretical technique can be used to derive a uniform-in-time bound of the numerical solution, in the associated functional norm, which in turn leads to a theoretical justification of the global-in-time energy estimate. The technical details will be reported in the future works.

\appendix
	
\section{Proof of Proposition~\ref{prop:prop 1}}  \label{appen: prop 1}

In the following analysis, we write $\displaystyle \sum_{\ell, m,n}$ to represent $\displaystyle \sum_{\ell, m,n=-K}^K$ for simplicity of notations unless the particular declarations.

To prove inequality \eqref{prop-1-0-5}, we begin with the summation-by-parts formula
\begin{equation}
  \| G^{(0)}_h \nabla_h \Delta_h f \|_2^2 = - \langle G^{(0)}_h  \Delta_h f ,
  G^{(0)}_h \Delta_h^2 f \rangle \le \| G^{(0)}_h  \Delta_h f \|_2 \cdot \| G^{(0)}_h \Delta_h^2 f \|_2 ,
  \label{prop-1-5-1}
\end{equation}
in which the Cauchy inequality has been applied in the last step. Meanwhile, applying the summation-by-parts formula again gives
\begin{equation}
  \| G^{(0)}_h \Delta_h^2 f \|_2^2 = - \langle G^{(0)}_h  \nabla_h \Delta_h f ,
  G^{(0)}_h \nabla_h \Delta_h^2 f \rangle \le \| G^{(0)}_h  \nabla_h \Delta_h f \|_2 \cdot \| G^{(0)}_h \nabla_h \Delta_h^2 f \|_2 .
  \label{prop-1-5-2}
\end{equation}
In turn, a substitution of~\eqref{prop-1-5-2} into \eqref{prop-1-5-1} results in
\begin{equation*}
   \| G^{(0)}_h \nabla_h \Delta_h f \|_2^2
   \le \| G^{(0)}_h  \Delta_h f \|_2 \cdot \| G^{(0)}_h  \nabla_h \Delta_h f \|_2^\frac12   \cdot \| G^{(0)}_h \nabla_h \Delta_h^2 f \|_2^\frac12,
\end{equation*}
so that
\begin{equation*}
    \| G^{(0)}_h \nabla_h \Delta_h f \|_2^\frac32
   \le \| G^{(0)}_h  \Delta_h f \|_2  \cdot \| G^{(0)}_h \nabla_h \Delta_h^2 f \|_2^\frac12 ,
\end{equation*}
which then leads to \eqref{prop-1-0-5}.

In terms of inequality~\eqref{prop-1-0-3}, we make use of the following identity
\begin{equation*}
  \| \Delta_h f \|_2^2  =  L^3 \sum_{\ell, m,n}  \lambda_{\ell,m,n}^2
   \cdot | \hat{f}_{\ell,m,n} |^2 . 
\end{equation*}
An application of Parseval equality to the discrete Fourier expansion of $G^{(5)}_h f$, given by~\eqref{Fourier-7-1}, yields
\begin{equation}
    \| G^{(5)}_h f \|_2^2  = L^3 \sum_{\ell, m,n}  \frac{1 - {\rm e}^{- \tau \Lambda_{\ell,m,n}} }{\tau}   \lambda_{\ell,m,n}^2 \cdot | \hat{f}_{\ell,m,n} |^2 .  \label{prop-1-1-4}
\end{equation}
Then, the following inequality is valid:
\begin{equation*}
    \tau \| G^{(5)}_h f \|_2^2  =
    L^3 \sum_{\ell, m,n} ( 1 - {\rm e}^{- \tau \Lambda_{\ell,m,n}} )
     \lambda_{\ell,m,n}^2 \cdot | \hat{f}_{\ell,m,n} |^2
     \le L^3 \sum_{\ell, m,n}
     \lambda_{\ell,m,n}^2 \cdot | \hat{f}_{\ell,m,n} |^2 ,
\end{equation*}
in which the last step comes from a trivial fact that $1 - {\rm e}^{- \tau \Lambda_{\ell,m,n}} \le 1$. The proof of inequality~\eqref{prop-1-0-3} is completed.

Next, we show \eqref{prop-1-0-1}.
The applications of $G_h L_\kappa$ and $\Delta_h^2$ to the discrete Fourier expansion of $f \in \mathring{\mathcal C}_{\rm per}$, given by~\eqref{Fourier-1}, becomes
\begin{align}
  &
  ( G_h L_\kappa f )_{i,j,k} = \sum_{\ell, m,n} \frac{1
  - {\rm e}^{- \tau \Lambda_{\ell,m.n}} }{\tau \Lambda_{\ell,m,n}}
  \cdot \Lambda_{\ell,m,n}
  \hat{f}_{\ell,m,n} {\rm e}^{2 \pi \mathrm{i} ( \ell x_i + m y_j + n z_k)/L} ,   \label{prop-1-1-1}
\\
  &
  ( \Delta_h^2 f )_{i,j,k} = \sum_{\ell, m,n} \lambda_{\ell,m,n}^2
  \hat{f}_{\ell,m,n} {\rm e}^{2 \pi \mathrm{i} ( \ell x_i + m y_j + n z_k)/L} . \nonumber 
\end{align}
Subsequently, the discrete inner product  turns out to be
\begin{align}
  &
  \langle G_h L_\kappa f , \Delta_h^2 f \rangle = L^3 \sum_{\ell, m,n} \frac{1
  - {\rm e}^{- \tau \Lambda_{\ell,m.n}} }{\tau \Lambda_{\ell,m,n}}
  \cdot \Lambda_{\ell,m,n} \cdot \lambda_{\ell,m,n}^2 \cdot | \hat{f}_{\ell,m,n} |^2 .
   \label{prop-1-1-3}
\end{align}
In turn, a comparison between~\eqref{prop-1-1-3} and \eqref{prop-1-1-4} results in the first equality in~\eqref{prop-1-0-1}. To prove the second inequality in~\eqref{prop-1-0-1}, we begin with the following observation:
\begin{equation*}
\begin{aligned}
  (-1 + \lambda_{\ell,m,n} )^2 + \kappa & = 2 - 2 \lambda_{\ell,m,n} + \frac12 \lambda_{\ell,m,n}^2
  + \frac12 \lambda_{\ell,m,n}^2 + (\kappa -1)
\\
  &
  = \frac12 ( 2 - \lambda_{\ell,m,n} )^2
  + \frac12 \lambda_{\ell,m,n}^2 + (\kappa -1)
  \ge \frac12 \lambda_{\ell,m,n}^2 + (\kappa -1) ,
\end{aligned}
\end{equation*}
so that
\begin{equation*}
\begin{aligned}
&   \Lambda_{\ell, m,n} = \Big( (-1 + \lambda_{\ell,m,n} )^2 + \kappa \Big) \lambda_{\ell,m,n}
   \ge \frac12 \lambda_{\ell,m,n}^3 + (\kappa -1) \lambda_{\ell,m,n} ,
\\
  &
   \Lambda_{\ell, m,n} \cdot \lambda_{\ell,m,n}^2
   \ge \frac12 \lambda_{\ell,m,n}^5 + (\kappa -1) \lambda_{\ell,m,n}^3 ,
\end{aligned}
\end{equation*}
which in turn leads to the following inequality
\begin{align}
    \| G^{(5)}_h f \|_2^2 & =  L^3 \sum_{\ell, m,n}  \frac{1 - {\rm e}^{- \tau \Lambda_{\ell,m,n}} }{\tau \Lambda_{\ell, m,n}}  \Lambda_{\ell, m,n} \cdot  \lambda_{\ell,m,n}^2 \cdot | \hat{f}_{\ell,m,n} |^2  \label{prop-1-1-6}
 \\
 &  \ge
    L^3 \sum_{\ell, m,n} \frac{1 - {\rm e}^{- \tau \Lambda_{\ell,m,n}} }{\tau \Lambda_{\ell, m,n}}  \Big( \frac12 \lambda_{\ell,m,n}^5 + (\kappa -1) \lambda_{\ell,m,n}^3  \Big) | \hat{f}_{\ell,m,n} |^2 . \nonumber
\end{align}

An application of Parseval equality to the discrete Fourier expansion of $G^{(0)}_h \nabla_h \Delta_h^2 f$ and $G^{(0)}_h \nabla_h \Delta_h f$ indicates that
\begin{align}
  &
    \| G^{(0)}_h \nabla_h \Delta_h^2 f \|_2^2  = L^3\sum_{\ell, m,n} \frac{1 - {\rm e}^{- \tau \Lambda_{\ell,m,n}} }{\tau \Lambda_{\ell,m,n}}  \cdot \lambda_{\ell,m,n}^5
    \cdot | \hat{f}_{\ell,m,n} |^2 ,\label{prop-1-1-7}
\\
  &
    \| G^{(0)}_h \nabla_h \Delta_h f \|_2^2  = L^3 \sum_{\ell, m,n}  \frac{1 - {\rm e}^{- \tau \Lambda_{\ell,m,n}} }{\tau \Lambda_{\ell,m,n}}  \cdot \lambda_{\ell,m,n}^3
    \cdot | \hat{f}_{\ell,m,n} |^2 . \label{prop-1-1-8}
\end{align}
Similarly, a comparison between~\eqref{prop-1-1-6}, \eqref{prop-1-1-7}-\eqref{prop-1-1-8} leads to the second inequality of~\eqref{prop-1-0-1}.

To prove \eqref{prop-1-0-2}, we see that the discrete Fourier expansion of $\Delta_h^2 {\rm e}^{- \tau L_\kappa} f $ can be represented as
\begin{equation*}
  ( \Delta_h^2 {\rm e}^{- \tau L_\kappa} f)_{i,j,k} = \sum_{\ell, m,n}  {\rm e}^{- \tau \Lambda_{\ell,m.n}}  \cdot \lambda_{\ell,m,n}^2
  \hat{f}_{\ell,m,n} {\rm e}^{2 \pi \mathrm{i} ( \ell x_i + m y_j + n z_k)/L} ,
\end{equation*}
and its discrete inner product with $G_h L_\kappa  f$, in which the discrete Fourier expansion is given by~\eqref{prop-1-1-1}, turns out to be
\begin{equation}
 \left\langle G_h L_\kappa  f , \Delta_h^2 {\rm e}^{- \tau L_\kappa} f  \right\rangle
 = L^3 \sum_{\ell, m,n} \frac{1
  - {\rm e}^{- \tau \Lambda_{\ell,m.n}} }{\tau \Lambda_{\ell,m,n}}
  \cdot \Lambda_{\ell,m,n} \cdot {\rm e}^{- \tau \Lambda_{\ell,m.n}}  \cdot \lambda_{\ell,m,n}^2 | \hat{f}_{\ell,m,n} |^2 ,   \label{prop-1-2-2}
\end{equation}
On the other hand, by the first equality in~\eqref{prop-1-1-6}, we get
\begin{equation}
    \| G^{(5)}_h {\rm e}^{- \tau L_\kappa} f \|_2^2  =  L^3 \sum_{\ell, m,n}  \frac{1 - {\rm e}^{- \tau \Lambda_{\ell,m,n}} }{\tau \Lambda_{\ell, m,n}}  \Lambda_{\ell, m,n} \cdot  \lambda_{\ell,m,n}^2 \cdot {\rm e}^{- 2 \tau \Lambda_{\ell,m,n}} | \hat{f}_{\ell,m,n} |^2 .
     \label{prop-1-2-3}
\end{equation}
Again, a comparison between \eqref{prop-1-2-2} and \eqref{prop-1-2-3} implies inequality~\eqref{prop-1-0-2}, because of the fact that $| {\rm e}^{- \tau \Lambda_{\ell,m.n}}  | \le 1$.

Inequalities in~\eqref{prop-1-0-4} and \eqref{prop-1-0-6} comes directly from the estimates in Lemma~\ref{lem:lem 1}; the details are skipped for the sake of brevity.

\section{Proof of Proposition~\ref{prop:prop 2}} \label{appen: prop 2}

A discrete Fourier expansion~\eqref{Fourier-1} is assumed for $f \in {\mathcal C}_{\rm per}$, and we set the corresponding expansion for $g \in {\mathcal C}_{\rm per}$ as
\begin{equation*}
  g_{i,j,k} = \sum_{\ell, m,n} \hat{g}_{\ell, m,n} {\rm e}^{2 \pi \mathrm{i} ( \ell x_i + m y_j + n z_k)/L} .
\end{equation*}
Subsequently, the discrete Fourier expansion of $g -  {\rm e}^{-\tau L_\kappa} f$ turns out to be
\begin{equation*}
  ( g -  {\rm e}^{- \tau L_\kappa} f)_{i,j,k} = \sum_{\ell, m,n} ( \hat{g}_{\ell, m,n}
  - {\rm e}^{-\tau \Lambda_{\ell, m,n}} \hat{f}_{\ell, m,n} )
  {\rm e}^{2 \pi \mathrm{i} ( \ell x_i + m y_j + n z_k)/L} . 
\end{equation*}
Of course, its discrete $H^2$ norm becomes
\begin{equation}
  \| \Delta_h ( g -  {\rm e}^{-\tau L_\kappa} f ) \|_2^2
  =  L^3 \sum_{\ell, m,n} \lambda_{\ell, m,n}^2
   \cdot | \hat{g}_{\ell, m,n} - {\rm e}^{-\tau \Lambda_{\ell, m,n}} \hat{f}_{\ell, m,n} |^2 .
   \label{prop-2-1}
\end{equation}
In turn, a combination of the representation formula~\eqref{prop-1-2-2} and \eqref{prop-2-1} yields
\begin{equation*}
\begin{aligned}
  & \quad~
  \tau \left\langle G_h L_\kappa  f , \Delta_h^2 {\rm e}^{-\tau L_\kappa} f  \right\rangle
   + \| \Delta_h ( g -  {\rm e}^{-\tau L_\kappa}  f)  \|_2^2
\\
  & =
  L^3 \sum_{\ell, m,n} \Big(  ( 1 - {\rm e}^{- \tau \Lambda_{\ell, m,n}} )
   {\rm e}^{-\tau \Lambda_{\ell, m,n}}  \lambda_{\ell, m,n}^2   | \hat{f}_{\ell, m,n} |^2
  + \lambda_{\ell, m,n}^2   | \hat{g}_{\ell, m,n} - {\rm e}^{-\tau L_\kappa} \hat{f}_{\ell, m,n} |^2 \Big)
\\
  & =
  L^3 \sum_{\ell, m,n} \lambda_{\ell, m,n}^2 \Big(  ( 1 - {\rm e}^{-\tau \Lambda_{\ell, m,n}} )
   {\rm e}^{\tau \Lambda_{\ell, m,n}}   | {\rm e}^{-\tau \Lambda_{\ell, m,n}} \hat{f}_{\ell, m,n} |^2
    +  | \hat{g}_{\ell, m,n} - {\rm e}^{-\tau L_\kappa} \hat{f}_{\ell, m,n} |^2 \Big)
\\
  & =
  L^3 \sum_{\ell, m,n} \lambda_{\ell, m,n}^2 ( 1 - {\rm e}^{-\tau \Lambda_{\ell, m,n}} )
  \Big(  {\rm e}^{\tau \Lambda_{\ell, m,n}}
    |  {\rm e}^{-\tau \Lambda_{\ell, m,n}} \hat{f}_{\ell, m,n} |^2
\\
  &  \qquad \qquad \qquad \qquad \qquad \qquad \qquad
    +  \frac{1}{1 - {\rm e}^{-\tau \Lambda_{\ell, m,n}}}
     | \hat{g}_{\ell, m,n} - {\rm e}^{-\tau L_\kappa} \hat{f}_{\ell, m,n} |^2 \Big) .
\end{aligned}
\end{equation*}
The following lower bound can be derived for each fixed mode frequency $(\ell, m,n)$:
\begin{equation*}
\begin{aligned}
    {\rm e}^{\tau \Lambda_{\ell, m,n}}
    a^2  +  \frac{1}{1 - {\rm e}^{-\tau \Lambda_{\ell. m,n}}} b^2
   & = a^2 + b^2 + (  {\rm e}^{\tau \Lambda_{\ell, m,n}}  -1 )
    a^2  +  \Big( \frac{1}{1 - {\rm e}^{-\tau \Lambda_{\ell, m,n}}} -1 \Big) b^2
\\
  &
  \ge a^2 + b^2 + 2 ab = ( a + b)^2 ,  \quad \forall \, a , b \ge 0 ,
\end{aligned}
\end{equation*}
where the Cauchy inequality has been applied in the second step. Then we see that
\begin{equation*}
\begin{aligned}
  &
  {\rm e}^{\tau \Lambda_{\ell, m,n}}
    | {\rm e}^{-\tau \Lambda_{\ell, m,n}} \hat{f}_{\ell, m,n} |^2
    +  \frac{1}{1 - {\rm e}^{-\tau \Lambda_{\ell, m,n}}}
     | \hat{g}_{\ell, m,n} - {\rm e}^{-\tau \Lambda_{\ell, m,n}} \hat{f}_{\ell, m,n} |^2
\\
  & \qquad \ge
  \Big( | {\rm e}^{-\tau \Lambda_{\ell, m,n}} \hat{f}_{\ell, m,n} |
  + | \hat{g}_{\ell, m,n} - {\rm e}^{-\tau \Lambda_{\ell, m,n}} \hat{f}_{\ell, m,n} | \Big)^2
  \ge | \hat{g}_{\ell, m,n} |^2 ,
\end{aligned}
\end{equation*}
so that
\begin{equation*}
   \tau \left\langle G_h L_\kappa  f , \Delta_h^2 {\rm e}^{-\tau L_\kappa} f  \right\rangle
   + \| \Delta_h ( g -  {\rm e}^{-\tau L_\kappa}  f ) \|_2^2
  \ge L^3 \sum_{\ell, m,n} \lambda_{\ell, m,n}^2 ( 1 - {\rm e}^{-\tau \Lambda_{\ell, m,n}} )
  | \hat{g}_{\ell, m,n} |^2 .
\end{equation*}
In comparison with the representation formula for $\tau \| G_h^{(5)} g \|_2^2$:
\begin{align*}
    \tau \| G^{(5)}_h g \|_2^2  & = \tau L^3 \sum_{\ell, m,n}  \frac{1 - {\rm e}^{- \tau \Lambda_{\ell, m,n} } }{\tau \Lambda_{\ell, m,n}}  \Lambda_{\ell, m,n} \cdot  \lambda_{\ell, m,n}^2 \cdot | \hat{g}_{\ell, m,n} |^2
\\
  & =
     L^3 \sum_{\ell, m,n} \lambda_{\ell, m,n}^2 ( 1 - {\rm e}^{-\tau \Lambda_{\ell, m,n}} )  | \hat{g}_{\ell, m,n} |^2 ,
\end{align*}
we see that~\eqref{prop-2-0} has been established. This finishes the proof of Proposition~\ref{prop:prop 2}.

\bibliographystyle{amsplain}
\bibliography{draft1.bib}

\end{document}